\documentclass[12pt]{amsart}

\usepackage{amsmath,amsfonts,amssymb,latexsym,mathtools}
\usepackage{enumitem}
\usepackage[usenames, dvipsnames]{xcolor}
\usepackage{hyperref}
\usepackage{ytableau}
\usepackage{centernot}
\usepackage{bbding}
\usepackage{tikz-cd}
\usepackage{tikz}
\usetikzlibrary{calc, shapes, backgrounds,arrows,positioning,plotmarks}
\tikzset{>=stealth',
  head/.style = {fill = white, text=black},
  plaque/.style = {draw, rectangle, minimum size = 10mm}, 
  pil/.style={->,thick},
  junct/.style = {draw,circle,inner sep=0.5pt,outer sep=0pt, fill=black}
  }
  
\newtheorem{theorem}{Theorem}[section]
\newtheorem{lemma}[theorem]{Lemma}
\newtheorem{proposition}[theorem]{Proposition}
\newtheorem{corollary}[theorem]{Corollary}
\newtheorem{conjecture}[theorem]{Conjecture}

\theoremstyle{definition}
\newtheorem{definition}[theorem]{Definition}
\newtheorem{examplex}[theorem]{Example}
\newenvironment{example}
  {\pushQED{\qed}\examplex}
  {\popQED\endexamplex}

\theoremstyle{remark}

\numberwithin{equation}{section}

\usepackage[colorinlistoftodos]{todonotes}

\newcommand{\flags}{\ensuremath{\mathcal{F}}}
\newcommand{\GL}{\ensuremath{GL}}

\newcommand{\downshift}{\downarrow \!}
\newcommand{\rightshift}{\rightarrow \!}

\newcommand{\xx}{\ensuremath{\mathbf{x}}}
\newcommand{\Ess}{\ensuremath{\mathrm{Ess}}}
\newcommand{\Dom}{\ensuremath{\mathrm{Dom}}}
\newcommand{\Esshat}{\ensuremath{\mathrm{Ess}^\prime}}
\DeclareMathOperator{\Spec}{\ensuremath{\mathrm{Spec}}}

\newcommand{\wt}{\ensuremath{\mathrm{wt}}}

\newcommand{\Schub}{\ensuremath{\mathfrak{S}}}

\newcommand{\bpd}[1]{{\sf BPD}(#1)}
\newcommand{\init}[1]{{\tt init}(#1)}
\newcommand{\rank}[1]{{\tt rank}(#1)}

\begin{document}


\title[Gr\"obner geometry of Schubert polynomials through ice]{Gr\"obner geometry of Schubert polynomials \\ through ice}  

\author[Z. Hamaker]{Zachary Hamaker}
\address[ZH]{Department of Mathematics, University of Florida, Gainesville, FL 32601}
\email{zhamaker@ufl.edu}

\author[O. Pechenik]{Oliver Pechenik}
\address[OP]{Department of Mathematics, University of Michigan, Ann Arbor, MI 48109}
\email{pechenik@umich.edu}

\author[A. Weigandt]{Anna Weigandt}
\address[AW]{Department of Mathematics, University of Michigan, Ann Arbor, MI 48109}
\email{weigandt@umich.edu}


\date{\today}


\keywords{}

\begin{abstract}
The geometric naturality of Schubert
polynomials and their combinatorial pipe dream representations was established by
Knutson and Miller (2005) via antidiagonal Gr\"obner degeneration of matrix
Schubert varieties. We consider instead diagonal Gr\"obner degenerations. In
this dual setting, Knutson, Miller, and Yong (2009) obtained alternative combinatorics
for the class of ``vexillary'' matrix Schubert varieties. We
initiate a study of general diagonal degenerations, relating them to a neglected formula of
Lascoux (2002) in terms of the $6$-vertex ice model (recently rediscovered by
Lam, Lee, and Shimozono (2018) in the guise of ``bumpless pipe dreams'').
\end{abstract}

\maketitle

\section{Introduction}
\label{s:intro}

Let $\flags_n$ be the complex flag variety, the parameter space for complete flags of nested vector subspaces of $\mathbb{C}^n$.
The Schubert cell decomposition of $\flags_n$ yields a distinguished $\mathbb{Z}$-linear basis for the cohomology ring $H^\star(\flags_n)$.
On the other hand, A.~Borel \cite{Borel} presented this ring as
\[
H^\star(\flags_n) \cong \mathbb{Z}[x_1, \dots, x_n] / I,
\]
where $I$ is the ideal generated by the nonconstant elementary symmetric polynomials. 

It is natural to desire polynomial representatives for the Schubert basis with respect to this presentation. Building on work of I.~Bernstein, I.~Gelfand, and S.~Gelfand \cite{Bernstein.Gelfand.Gelfand}, A.~Lascoux and M.-P.~Sch\"utzenberger \cite{Lascoux.Schutzenberger} introduced \emph{Schubert polynomials}.
These are combinatorially well-adapted coset representatives for images of Schubert cohomology classes under the Borel isomorphism. 
In fact, Lascoux and Sch\"utzenberger  introduced more general \emph{double Schubert polynomials} that represent Schubert classes in the $T$-equivariant cohomology of $\flags_n$ (where $T \subset \GL_n(\mathbb{C})$ is the group of invertible diagonal matrices).

Since their introduction, (double) Schubert polynomials have become central objects in algebraic combinatorics (see, e.g., \cite{Billey.Jockusch.Stanley, Fomin.Stanley, Hamaker.Pechenik.Speyer.Weigandt, Lenart, Macdonald:notes}). 
They have also been interpreted through the geometry of degeneracy loci and used to unify many classical results in that area \cite{Fulton, Fulton.Pragacz}. 
A.~Knutson and E.~Miller \cite{Knutson.Miller} gave an alternative geometric justification for the naturality of Schubert polynomials by Gr\"obner degeneration of certain affine varieties.
Moreover, they recovered aspects of the combinatorics of Schubert polynomials through this geometry, including identifying irreducible components of the degeneration with the  \emph{pipe dreams} of earlier combinatorial formulas \cite{Bergeron.Billey,Fomin.Kirillov}. This explicit degeneration demonstrates the geometric naturality of pipe dream combinatorics.

Lascoux \cite{Lascoux:ice} introduced an alternate combinatorial model for (double) Schubert polynomials using  states of the square-ice (``$6$-vertex'') model from statistical physics. (For background and history of these ideas, see, e.g., \cite{Baxter, Bressoud, Elkies.Kuperberg.Larsen.Propp,Kuperberg, Robbins.Rumsey}.)
Recently, T.~Lam, S.-J.~Lee, and M.~Shimozono~\cite{Lam.Lee.Shimozono} rediscovered this Schubert polynomial model and gave a cleaner description in terms of \emph{bumpless pipe dreams}.
The connection between \cite{Lam.Lee.Shimozono} and \cite{Lascoux:ice} is detailed in~\cite{Weigandt}. 

Although both ordinary pipe dreams and bumpless pipe dreams compute the same double Schubert polynomials and appear superficially similar, they compute these polynomials in fundamentally different ways. In particular, (except in trivial cases) no weight-preserving bijection exists between these two sets. In light of this fact, the geometric content of bumpless pipe dreams and Lascoux's ice formula remains unclear.

\begin{example}
Let $w$ be the permutation $2143 \in S_4$. The three ordinary pipe dreams
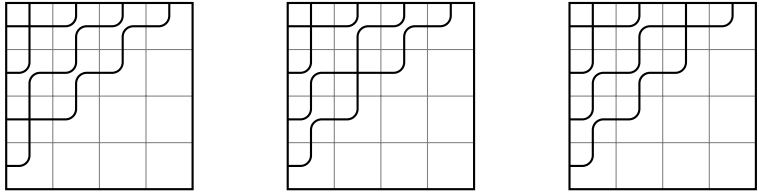
\[\begin{tikzpicture}[x=1.5em,y=1.5em]
\draw[step=1,gray, thin] (0,1) grid (4,5);
\draw[color=black, thick](0,1)rectangle(4,5);
\draw[thick,rounded corners] (0,4.5)--(1.5,4.5)--(1.5,5);
\draw[thick,rounded corners] (0,3.5)--(.5,3.5)--(.5,5);
\draw[thick,rounded corners] (0,2.5)--(1.5,2.5)--(1.5,3.5)--(2.5,3.5)--(2.5,4.5)--(3.5,4.5)--(3.5,5);
\draw[thick,rounded corners](0,1.5)--(.5,1.5)--(.5,3.5)--(1.5,3.5)--(1.5,4.5)--(2.5,4.5)--(2.5,5);
\end{tikzpicture}
\hspace{3em}
\begin{tikzpicture}[x=1.5em,y=1.5em]
\draw[step=1,gray, thin] (0,1) grid (4,5);
\draw[color=black, thick](0,1)rectangle(4,5);
\draw[thick,rounded corners] (0,4.5)--(1.5,4.5)--(1.5,5);
\draw[thick,rounded corners] (0,3.5)--(.5,3.5)--(.5,5);
\draw[thick,rounded corners] (0,2.5)--(.5,2.5)--(.5,3.5)--(2.5,3.5)--(2.5,4.5)--(3.5,4.5)--(3.5,5);
\draw[thick,rounded corners](0,1.5)--(.5,1.5)--(.5,2.5)--(1.5,2.5)--(1.5,4.5)--(2.5,4.5)--(2.5,5);
\end{tikzpicture}
\hspace{3em}
\begin{tikzpicture}[x=1.5em,y=1.5em]
\draw[step=1,gray, thin] (0,1) grid (4,5);
\draw[color=black, thick](0,1)rectangle(4,5);
\draw[thick,rounded corners] (0,4.5)--(1.5,4.5)--(1.5,5);
\draw[thick,rounded corners] (0,3.5)--(.5,3.5)--(.5,5);
\draw[thick,rounded corners] (0,2.5)--(.5,2.5)--(.5,3.5)--(1.5,3.5)--(1.5,4.5)--(3.5,4.5)--(3.5,5);
\draw[thick,rounded corners](0,1.5)--(.5,1.5)--(.5,2.5)--(1.5,2.5)--(1.5,3.5)--(2.5,3.5)--(2.5,5);
\end{tikzpicture}
\]
 for this permutation
present the corresponding double Schubert polynomial as
\[\mathfrak S_w= (x_1-y_1)(x_3-y_1)+(x_1-y_1)(x_2-y_2)+(x_1-y_1)(x_1-y_3).\]
There are also three bumpless pipe dreams
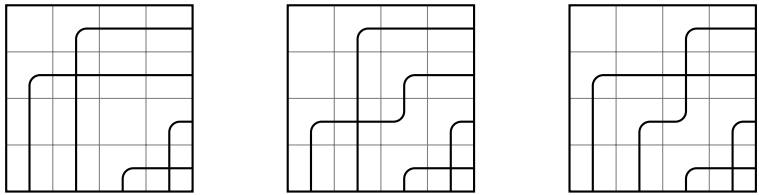
\[
\begin{tikzpicture}[x=1.5em,y=1.5em]
\draw[step=1,gray, thin] (0,1) grid (4,5);
\draw[color=black, thick](0,1)rectangle(4,5);
\draw[thick,rounded corners] (.5,1)--(.5,3.5)--(4,3.5);
\draw[thick,rounded corners] (1.5,1)--(1.5,4.5)--(4,4.5);
\draw[thick,rounded corners] (2.5,1)--(2.5,1.5)--(4,1.5);
\draw[thick,rounded corners] (3.5,1)--(3.5,2.5)--(4,2.5);
\end{tikzpicture}
\hspace{3em}
\begin{tikzpicture}[x=1.5em,y=1.5em]
\draw[step=1,gray, thin] (0,1) grid (4,5);
\draw[color=black, thick](0,1)rectangle(4,5);
\draw[thick,rounded corners] (.5,1)--(.5,2.5)--(2.5,2.5)--(2.5,3.5)--(4,3.5);
\draw[thick,rounded corners] (1.5,1)--(1.5,4.5)--(4,4.5);
\draw[thick,rounded corners] (2.5,1)--(2.5,1.5)--(4,1.5);
\draw[thick,rounded corners] (3.5,1)--(3.5,2.5)--(4,2.5);
\end{tikzpicture}
\hspace{3em}
\begin{tikzpicture}[x=1.5em,y=1.5em]
\draw[step=1,gray, thin] (0,1) grid (4,5);
\draw[color=black, thick](0,1)rectangle(4,5);
\draw[thick,rounded corners] (.5,1)--(.5,3.5)--(4,3.5);
\draw[thick,rounded corners] (1.5,1)--(1.5,2.5)--(2.5,2.5)--(2.5,4.5)--(4,4.5);
\draw[thick,rounded corners] (2.5,1)--(2.5,1.5)--(4,1.5);
\draw[thick,rounded corners] (3.5,1)--(3.5,2.5)--(4,2.5);
\end{tikzpicture}\]
for $w$. These give a presentation of the same double Schubert polynomial as
\[\mathfrak S_w= (x_1-y_1)(x_3-y_3)+(x_1-y_1)(x_2-y_1)+(x_1-y_1)(x_1-y_2).\]
Note that although these expressions are necessarily equal, this equality is only apparent after significant factoring and reorganizing. In particular, there is no weight-preserving way to match up the terms of the two summations.
\end{example}

In Lie-theoretic terms, one may identify $\flags_n$ with the homogeneous space $\GL_n(\mathbb{C}) /B$, where $B$ denotes the Borel subgroup of invertible upper triangular matrices.
Pulling back a Schubert cell in $\flags_n$ to $\GL_n(\mathbb{C})$, we may then consider its closure in the affine space of all $n \times n$ complex matrices.
W.~Fulton \cite{Fulton} showed that these \emph{matrix Schubert varieties} are irreducible, gave set-theoretic defining equations for them, and showed that these equations define reduced schemes. 
The key observation of Knutson and Miller is that these \emph{Fulton generators} form a Gr\"obner basis under any {\it antidiagonal} term order
(that is, any term order under which the initial term of each minor of a generic matrix is the product of the entries along its main antidiagonal).

It is at least as natural to consider the dual notion of {\it diagonal} term orders (that is, term orders where initial terms of minors are products along main diagonals). For example, much of the commutative algebra literature on determinantal ideals and generalizations focuses on this case (e.g., \cite{Sturmfels, Gonciulea.Miller, Bruns.Conca}).
Indeed, Knutson and Miller first tried unsuccessfully to carry out their program in this context before they realized that the antidiagonal term orders were more amenable to their approach.

The geometry of diagonal degenerations, in fact, is more complicated than the antidiagonal case.
In general, the Fulton generators are \emph{not} a Gr\"obner basis with respect to diagonal term orders. In \cite{Knutson.Miller.Yong}, it was shown that Fulton generators are diagonal Gr\"obner exactly for the class of matrix Schubert varieties called \emph{vexillary}. For general matrix Schubert varieties, the diagonal Gr\"obner degenerations can even fail to be reduced. Moreover, in the nonreduced case, different diagonal term orders can yield distinct scheme structures on the limiting space of the degeneration.

In this paper, we return to the diagonal setting. Despite the additional geometric complication, we propose that diagonal Gr\"obner degenerations naturally give rise to bumpless pipe dreams in an exactly analogous fashion to how antidiagonal degenerations yield ordinary pipe dreams.
Our main conjecture is the following:
\begin{conjecture}
\label{c:main}
	Let $\init{X_w}$ be the  Gr\"obner degeneration of a matrix Schubert variety with respect to any diagonal term order.
	The irreducible components of $\init{X_w}$, counted with multiplicities, naturally correspond to the bumpless pipe dreams for the permutation $w$.
\end{conjecture}

In particular, Conjecture~\ref{c:main} implies that, although different choices of diagonal term orders may yield degenerations to distinct schemes, the reduced irreducible components of the degeneration and their multiplicities do not depend on such a choice.
The vexillary case of Conjecture~\ref{c:main} follows from \cite{Knutson.Miller.Yong} and results in~\cite{Weigandt}.
Our main result is to prove Conjecture~\ref{c:main} for a larger class of permutations, called \emph{banner permutations}, extending the vexillary case. 
For these permutations, we are able to exhibit explicit diagonal Gr\"obner bases by modifying the Fulton generators in an appropriate fashion.
\begin{theorem}
\label{t:main}
	If $w$ is a banner permutation, then the CDG generators for $X_w$ are a diagonal Gr\"obner basis. The irreducible components of $\init{X_w}$, counted with multiplicities, naturally correspond to the bumpless pipe dreams for the permutation $w$.
\end{theorem}

The precise definition of \emph{banner permutations} appears in Sections~\ref{s:proof}, while the \emph{CDG generators} are defined for general $w$ in Section~\ref{s:cdg}.

The recursive arguments in \cite{Knutson.Miller} rely on the authors introducing and developing the combinatorics of a new \emph{mitosis} recursion for ordinary pipe dreams (see also, \cite{Miller:mitosis}). In contrast, bumpless pipe dreams appear well-adapted to the simpler and more classical \emph{transition} formula of Lascoux and Sch\"utzenberger \cite{Lascoux.Schutzenberger:tree} (see also, \cite{Macdonald:notes}). Our proof of Theorem~\ref{t:main} relies heavily on this latter recursion. Recently, Knutson \cite{Knutson:cotransition} has developed a dual notion of \emph{cotransition}, allowing him to simplify antidiagonal arguments of \cite{Knutson.Miller} in a similar fashion to the arguments here.

We believe that Theorem~\ref{t:main} holds in somewhat more generality than proved in this paper (see Conjecture~\ref{conj:patterns}) and we have hope that Theorem~\ref{t:main} can be thus extended using similar techniques to those employed here. However, we do not know a description of diagonal Gr\"obner bases in the most general case. Indeed, since different choices of diagonal term order can lead to different initial ideals, it is not guaranteed that there exists an explicit uniform description of Gr\"obner bases for all diagonal orders. Nonetheless, Conjecture~\ref{c:main} is supported by calculations in such cases. 
By computer, we have systematically verified Conjecture~\ref{c:main} through the symmetric group $S_7$ for one choice of diagonal term order, as well as in a variety of other experiments for larger permutations and for other diagonal term orders.

\ 

\noindent \textbf{Organization:} 
In Section~\ref{s:background}, we recall necessary background information. In Section~\ref{s:cdg}, we introduce generators  for Schubert determinantal ideals, which we call the \emph{CDG generators}.  These are a modification of the more standard Fulton generators. Section~\ref{s:block} introduces a block construction for partial permutations and develops its combinatorics. In Section~\ref{s:monomial}, we  introduce \emph{block predominant permutations} and establish a recurrence for certain monomial ideals constructed from CDG generators. We apply this recurrence in Section~\ref{s:proof} to prove Theorem~\ref{t:main}. In Section~\ref{s:future}, we make some conjectures and remarks regarding extensions and applications of Theorem~\ref{t:main}. 

\section{Background}
\label{s:background}
\subsection{Combinatorics of permutations}
Define  $[n] = \{1,2,\dots,n\}$ and let $S_n$ denote the symmetric group on $[n]$.
Each permutation $w \in S_n$ is determined by its one-line notation $w_1 w_2 \dots w_n$ where $w_i = w(i)$.
The \textbf{Rothe diagram} of $w$ is the set
\[
D_w = \{(i,j) \in [n] \times [n]: w(i) > j, w^{-1}(j) > i\}.
\]
We visualize $D_w$ as a subset of $[n] \times [n]$ by placing $\bullet$ in $(i,w(i))$ for each $i \in [n]$, then drawing lines below and to the right of each $\bullet$.
Then $D_w$ is the complement of the marked boxes.
For example, $D_{42153}$ is $\{(1,1),(1,2),(1,3),(2,1),(4,3)\}$, which can be visualized as
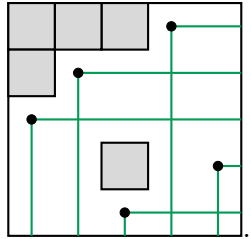
\[
  \begin{tikzpicture}[x=1.5em,y=1.5em]
      \draw[color=black, thick](0,1)rectangle(5,6);
     \filldraw[color=black, fill=gray!30, thick](0,5)rectangle(1,6);
     \filldraw[color=black, fill=gray!30, thick](0,4)rectangle(1,5);
      \filldraw[color=black, fill=gray!30, thick](2,5)rectangle(3,6);
     \filldraw[color=black, fill=gray!30, thick](1,5)rectangle(2,6);
     \filldraw[color=black, fill=gray!30, thick](2,2)rectangle(3,3);
     \draw[thick,ForestGreen] (5,5.5)--(3.5,5.5)--(3.5,1);
     \draw[thick,ForestGreen] (5,4.5)--(1.5,4.5)--(1.5,1);
     \draw[thick,ForestGreen] (5,3.5)--(.5,3.5)--(.5,1);
     \draw[thick,ForestGreen] (5,2.5)--(4.5,2.5)--(4.5,1);
     \draw[thick,ForestGreen] (5,1.5)--(2.5,1.5)--(2.5,1);
     \filldraw [black](3.5,5.5)circle(.1);
     \filldraw [black](1.5,4.5)circle(.1);
     \filldraw [black](.5,3.5)circle(.1);
     \filldraw [black](4.5,2.5)circle(.1);
     \filldraw [black](2.5,1.5)circle(.1);
     \end{tikzpicture}.
\]
The \textbf{essential set} of $w$ is 
\[
\Ess(w) = \{(i,j) \in D_w: (i+1,j), (i,j+1) \notin D_w\}.
\]
These are the maximally southeast cells in each connected component of $D_w$.  For instance, $\Ess(42153)=\{(1,3),(2,1),(4,3)\}$.

The $i$th \textbf{row} of $D_w$ is
\[
\{j\in [n]: (i,j) \in D_w\}.
\]
A permutation is \textbf{vexillary} if its rows are totally ordered by inclusion. For example, $42153$ is not vexillary, as neither of rows $2$ and $4$ is contained in the other.
The \textbf{Lehmer code} of a permutation $w$ is the sequence $c(w) = (c_1,\dots,c_n)$ where $c_i$ is the cardinality of the $i$th row of $D_w$. 
A permutation $w$ is \textbf{dominant} if $c(w)$ is weakly decreasing.
For example, $c(42153) = (3,1,0,1,0)$, so $42153$ is not dominant.

To every permutation $w \in S_n$, we associate a \textbf{rank function} $r_w : [n] \times [n] \to \mathbb{Z}$, where 
\[
r_w(i,j) = \# \{k \leq i : w(k) \leq j\}.
\]
For $v,w \in S_n$, we say $v \leq w$ in \textbf{Bruhat order} if $r_v(i,j) \geq r_w(i,j)$ for all $i,j \in [n]$.
We write $\lessdot$ for the covering relation in Bruhat order.

A {\bf partial permutation} is a $0$--$1$ matrix with at most one $1$ in each row and each column. The definitions of Rothe diagrams, essential sets, Lehmer codes, and rank functions naturally extend to partial permutations.   Let $M_{m,n}$ denote the set of $m \times n$ matrices over $\mathbb{C}$ and define $M_n \coloneqq M_{n,n}$. An $m\times n$ partial permutation $w\in M_{m,n}$ can be (uniquely) completed to a permutation matrix $\widetilde{w}\in M_{\max{\{m,n\}}}$.  This completion respects diagrams and essential sets.

\subsection{Matrix Schubert varieties}
Let $Z = (z_{ij})_{i \in [m], j\in [n]}$ be a matrix of distinct indeterminates and let $R = \mathbb{C}[Z]$.
We identify $M_{m,n}$ with the $mn$-dimensional affine space $\Spec R$.
For $A \in M_n$ and $I,J \subset [n]$, let $A_{I,J} = (a_{ij})_{i\in I, j\in J}$.
Then the \textbf{matrix Schubert variety} for $w \in S_n$ is the affine variety
\[
X_w = \left\{A \in M_n: \rank{A_{[i],[j]}} \leq r_w(i,j)\ \mbox{for all}\ i,j \in [n]\right\}.
\]
Let 
\[
I_w = \left\langle (r_w(i,j)+1)\text{-size minors in } Z_{[i],[j]} : i,j \in [n] \right\rangle \subseteq R
\]
be the {\bf Schubert determinantal ideal.}
It is easy to see that $X_w$ is the vanishing locus of the ideal $I_w$. Indeed, Fulton \cite[Proposition~3.3]{Fulton} showed that $I_w$ is prime, so
\[
X_w \cong \Spec R/I_w
\]
as reduced schemes. Moreover, he established that it is enough to consider the smaller generating set of $I_w$:
\begin{equation}\label{eq:fulton}
	I_w = \left\langle (r_w(i,j)+1)\text{-size minors in } Z_{[i],[j]} : (i,j) \in \Ess(w) \right\rangle.
\end{equation}
The minors in Equation~\eqref{eq:fulton} are called the \textbf{Fulton generators} of $I_w$.

For example, suppose $w = 42153$. Then the Fulton generators of $I_w$ are 
\begin{equation}\label{eq:Fulton_ex}
z_{11}, z_{12}, z_{13}, z_{21}, \begin{vmatrix}
z_{11} & z_{12} & z_{13} \\
z_{21} & z_{22} & z_{23} \\
z_{31} & z_{32} & z_{33}
\end{vmatrix},
\begin{vmatrix}
z_{11} & z_{12} & z_{13} \\
z_{21} & z_{22} & z_{23} \\
z_{41} & z_{42} & z_{43}
\end{vmatrix},
\begin{vmatrix}
z_{11} & z_{12} & z_{13} \\
z_{31} & z_{32} & z_{33} \\
z_{41} & z_{42} & z_{43}
\end{vmatrix},
\begin{vmatrix}
z_{21} & z_{22} & z_{23} \\
z_{31} & z_{32} & z_{33} \\
z_{41} & z_{42} & z_{43}
\end{vmatrix}.
\end{equation}

Following Equation~\eqref{eq:fulton}, we also define matrix Schubert varieties in $M_{m,n}$, indexed by partial permutations. 
 See \cite[Chapter~15]{miller2004combinatorial} for more details.

\subsection{Bumpless pipe dreams}

Following \cite{Lam.Lee.Shimozono}, a {\bf bumpless pipe dream} is a tiling of the $n\times n$ grid with the six tiles pictured below,
\begin{equation}
\label{eqn:pipes}
\raisebox{-.5em}{
	\begin{tikzpicture}[x=1.5em, y=1.5em]
	\draw[color=black, thick](0,1)rectangle(1,2);
	\draw[thick, rounded corners,ForestGreen] (.5,1)--(.5,1.5)--(1,1.5);
	\end{tikzpicture}
	\hspace{2em}
	\begin{tikzpicture}[x=1.5em, y=1.5em]
	\draw[color=black, thick](0,1)rectangle(1,2);
	\draw[thick, rounded corners,ForestGreen] (.5,2)--(.5,1.5)--(0,1.5);
	\end{tikzpicture}
	\hspace{2em}
	\begin{tikzpicture}[x=1.5em, y=1.5em]
	\draw[color=black, thick](0,1)rectangle(1,2);
	\draw[thick, rounded corners,ForestGreen] (0,1.5)--(1,1.5);
	\draw[thick, rounded corners,ForestGreen] (.5,1)--(.5,2);
	\end{tikzpicture}
	\hspace{2em}
	\begin{tikzpicture}[x=1.5em, y=1.5em]
	\draw[color=black, thick](0,1)rectangle(1,2);
	\end{tikzpicture}
	\hspace{2em}
	\begin{tikzpicture}[x=1.5em, y=1.5em]
	\draw[color=black, thick](0,1)rectangle(1,2);
	\draw[thick, rounded corners,ForestGreen] (0,1.5)--(1,1.5);
	\end{tikzpicture}
	\hspace{2em}
	\begin{tikzpicture}[x=1.5em, y=1.5em]
	\draw[color=black, thick](0,1)rectangle(1,2);
	\draw[thick, rounded corners,ForestGreen] (.5,1)--(.5,2);
	\end{tikzpicture}}
\end{equation}
so that there are $n$ pipes which
\begin{enumerate}
	\item start at the right edge of the grid,
	\item end at the bottom of the grid, and
	\item pairwise cross at most one time.
\end{enumerate}
If $P$ is a bumpless pipe dream, we define a permutation $w$ by setting $w(i)$ to be the column in which the $i$th pipe exits (labeling rows from top to bottom).
Write $\bpd{w}$ for the set of bumpless pipe dreams for $w$.  The ${\bf diagram}$ of $P$ is \[D(P):=\{(i,j): (i,j) \, \text{is a blank tile in} \, P\}.\] Each bumpless pipe dream has an associated weight
$\displaystyle \wt(P)=\prod_{(i,j)\in D(P)}(x_i-y_j)$.

Lam--Lee--Shimozono showed that the double Schubert polynomial $\mathfrak S_w(\mathbf x;\mathbf y)$ can be expressed as a sum over bumpless pipe dreams.
\begin{theorem}[{\cite[Theorem 5.13]{Lam.Lee.Shimozono}}]
	\label{thm:BPDformula}
\[ \mathfrak S_w(\mathbf x;\mathbf y)=\sum_{P\in \bpd{w}} \wt(P).\]
\end{theorem}
For our purposes, we take this theorem to be the definition of the double Schubert polynomial; the single Schubert polynomial is obtained from this by setting all $y$ variables to $0$.
For example, the bumpless pipe dreams for $w = 42153$ are (ignore the colors for now)
\begin{equation}\label{eq:bpd_ex}
  \begin{tikzpicture}[x=1.5em, y=1.5em]
  \draw[step=1, gray, thin] (0,1) grid (5,6);
      \draw[color=black, thick](0,1)rectangle(5,6);
     \draw[thick, rounded corners,ForestGreen] (5,5.5)--(3.5,5.5)--(3.5,1);
     \draw[ultra thick, rounded corners,Red] (5,4.5)--(1.5,4.5)--(1.5,1);
     \draw[ultra thick, rounded corners,Blue] (5,3.5)--(.5,3.5)--(.5,1);
     \draw[thick, rounded corners,ForestGreen] (5,2.5)--(4.5,2.5)--(4.5,1);
     \draw[thick, rounded corners,ForestGreen] (5,1.5)--(2.5,1.5)--(2.5,1);
     \end{tikzpicture}
\hspace{2em}
  \begin{tikzpicture}[x=1.5em, y=1.5em]
    \draw[step=1, gray, thin] (0,1) grid (5,6);
      \draw[color=black, thick](0,1)rectangle(5,6);
     \draw[thick, rounded corners,ForestGreen] (5,5.5)--(3.5,5.5)--(3.5,1);
     \draw[thick, rounded corners,ForestGreen] (5,4.5)--(1.5,4.5)--(1.5,1);
     \draw[ultra thick, rounded corners,Blue] (5,3.5)--(2.5,3.5)--(2.5,2.5)--(.5,2.5)--(.5,1);
     \draw[thick, rounded corners,ForestGreen] (5,2.5)--(4.5,2.5)--(4.5,1);
     \draw[thick, rounded corners,ForestGreen] (5,1.5)--(2.5,1.5)--(2.5,1);
     \end{tikzpicture}
\hspace{2em}
  \begin{tikzpicture}[x=1.5em, y=1.5em]
    \draw[step=1, gray, thin] (0,1) grid (5,6);
      \draw[color=black, thick](0,1)rectangle(5,6);
     \draw[thick, rounded corners,ForestGreen] (5,5.5)--(3.5,5.5)--(3.5,1);
   
     \draw[thick, rounded corners,ForestGreen] (5,3.5)--(.5,3.5)--(.5,1);
     \draw[thick, rounded corners,ForestGreen] (5,2.5)--(4.5,2.5)--(4.5,1);
     \draw[thick, rounded corners,ForestGreen] (5,1.5)--(2.5,1.5)--(2.5,1);
  \draw[ultra thick, rounded corners,Red] (5,4.5)--(2.5,4.5)--(2.5,2.5)--(1.5,2.5)--(1.5,1);
     \end{tikzpicture}.
\end{equation}
Hence, 
\[
\Schub_{42153}(\mathbf x;\mathbf y) = (x_1 - y_1)(x_1 - y_2) (x_1-y_3)(x_2-y_1)\bigg(  (x_4-y_3) + (x_3-y_1) + (x_2-y_2)  \bigg).
\]


The {\bf Rothe bumpless pipe dream} of $w$ is the (unique) bumpless pipe dream $P_w$ that has a $\begin{tikzpicture}[scale=.5][x=1.5em, y=1.5em]
	\draw[color=black, thick](0,1)rectangle(1,2);
	\draw[thick, rounded corners,ForestGreen] (.5,1)--(.5,1.5)--(1,1.5);
	\end{tikzpicture}$ tile in position $(i, w(i))$ for all $i$ and contains no $\begin{tikzpicture}[scale=.5][x=1.5em, y=1.5em]
	\draw[color=black, thick](0,1)rectangle(1,2);
	\draw[thick, rounded corners,ForestGreen] (.5,2)--(.5,1.5)--(0,1.5);
	\end{tikzpicture}$ tiles. 
	It is the only bumpless pipe dream $P \in \bpd w$ satisfying $D(P) = D_w$.
For example, the first bumpless pipe dream in \eqref{eq:bpd_ex} is the Rothe bumpless pipe dream of $42153$.

There are natural local moves on bumpless pipe dreams called \textbf{droops} that preserve the permutation.
A droop is performed on a pair $\begin{tikzpicture}[scale=.5][x=1.5em, y=1.5em]
	\draw[color=black, thick](0,1)rectangle(1,2);
	\draw[thick, rounded corners,ForestGreen] (.5,1)--(.5,1.5)--(1,1.5);
	\end{tikzpicture}$ at $(i,j)$ and $\begin{tikzpicture}[scale=.5][x=1.5em, y=1.5em]
	\draw[color=black, thick](0,1)rectangle(1,2);
	\end{tikzpicture}$ at $(k,\ell)$ where $i < k, j < \ell$ by placing $\begin{tikzpicture}[scale=.5][x=1.5em, y=1.5em]
	\draw[color=black, thick](0,1)rectangle(1,2);
	\end{tikzpicture}$ at $(i,j)$, placing  $\begin{tikzpicture}[scale=.5][x=1.5em, y=1.5em]
	\draw[color=black, thick](0,1)rectangle(1,2);
	\draw[thick, rounded corners,ForestGreen] (.5,2)--(.5,1.5)--(0,1.5);
	\end{tikzpicture}$ at $(k,\ell)$ and modifying the pipe originally passing through $(i,j)$ so that it passes through $(k,\ell)$ instead. A droop is permissible if  $(i,j)$ is the only place a pipe bends within the rectangle $[i,k]\times[j,\ell]$.
	For an example, see Equation~\eqref{eq:bpd_ex}, where the bolded blue and red pipes in the diagram on the left correspond to available droops.

\begin{proposition}[{\cite[Proposition 5.3]{Lam.Lee.Shimozono}}]
\label{p:droop}
Let $w$ be a permutation.
Every $P \in \bpd w$ can be obtained from $P_w$ by a sequence of droops.
\end{proposition}

We will also need to consider bumpless pipe dreams for partial permutations.  Let $w\in M_{m,n}$ be a partial permutation and $\widetilde{w}$ its completion to a permutation. We define \[{\sf BPD}(w)=\{P\mid_{m\times n}:P\in {\sf BPD}(\widetilde{w})\},\] where $P\mid_{m\times n}$ denotes the restriction of $P$ to its first $m$ rows and $n$ columns.
Note droops only modify positions weakly northwest of cells in the Rothe diagram of $w$.
Therefore, Proposition~\ref{p:droop} shows we can reconstruct $P$ from $P\mid_{m \times n}$ since they are connected to $P_w$ and $P_w\mid_{m \times n}$, respectively, by the same sequence of droops.

\subsection{The transition formula}
(Double) Schubert polynomials satisfy a recurrence called \textbf{transition}.
Let $t_{ij}$ be the transposition $(i\; j) \in S_n$.
For $v \in S_n$ and $r \in [n]$, we define
\[
I(v,r) = \{i < r: v \lessdot vt_{ir}\} \quad \mbox{and} \quad \Phi(v,r) = \{ vt_{ir} : i \in I(v,r)\}.
\]

An {\bf inversion} in $w \in S_n$ is a pair $(i,j)$ such that $i < j$ and $w(i) > w(j)$. {\bf Lexicographic order} on inversions of $w$ is given by $(i_1, j_1) > (i_2, j_2)$ if $i_1 > i_2$ or if $i_1 = i_2$ and $j_1 > j_2$.

\begin{theorem}[Equivariant Transition, {\cite[Proposition 4.1]{KohnertVeigneau}}\footnote{We believe this result was known by experts prior, but we are unaware of  any explicit earlier reference in the literature. The ordinary cohomology case appeared first in~\cite{Lascoux.Schutzenberger:tree}.}]
	\label{t:LS_tree}
	
	Let $w \in S_n$ with lexicographically largest inversion $(r,w^{-1}(s))$ and let $v \coloneqq  w t_{rw^{-1}(s)}.$
	Then $v \lessdot w$ and 
	\[
	\mathfrak{S}_w = (x_r - y_s)\mathfrak{S}_v  + \sum_{u \in \Phi(v,r)} \mathfrak{S}_u.
	\]
\end{theorem}

This result is a straightforward consequence of the equivariant Monk's rule, which determines the equivariant cohomology of $\flags_n$.

The combinatorics of bumpless pipe dreams is compatible with transition.

\begin{lemma}
	\label{lemma:diagramtransition}
	There is a bijection  \[\Psi:\bpd{v}\cup \bigcup_{u\in \Phi(v,r)} \bpd{u}\rightarrow \bpd{w}\] so that
\[D(\Psi(P))=\begin{cases}
D(P)\cup\{(r,s)\}& \text{if}\, P\in \bpd{v} \, \text{and}\\
D(P) &\text{otherwise}.
\end{cases}\]
\end{lemma}
\begin{proof}
This follows by restricting the bijection in \cite[Proposition~5.2]{Weigandt} to reduced bumpless pipe dreams.
\end{proof}

Continuing our running example $w = 42153$, the lexicographically largest inversion is $(r,w^{-1}(s)) = (4,5)$, so we have $v = w t_{45} = 42135$. Since 
\[
\Phi(v,4) = \{ u^{(1)} = 43125, u^{(2)} = 42315 \},
\]
Lemma~\ref{lemma:diagramtransition} claims a bijection between $\bpd{w}$ and the unions of $\bpd{u^{(1)}}$, $\bpd{u^{(2)}}$, and $\bpd{v}$. Indeed, in this case, each of these three permutations $u^{(1)}, u^{(2)}, v$ is dominant and has a unique bumpless pipe dream:

\begin{equation}\label{eq:trans_ex}
 v \colon 
  \begin{tikzpicture}[x=1.5em,y=1.5em]
  \draw[step=1,gray, thin] (0,1) grid (5,6);
      \draw[color=black, thick](0,1)rectangle(5,6);
     \draw[thick,rounded corners,ForestGreen] (5,5.5)--(3.5,5.5)--(3.5,1);
     \draw[thick,rounded corners,ForestGreen] (5,4.5)--(1.5,4.5)--(1.5,1);
     \draw[thick,rounded corners,ForestGreen] (5,3.5)--(.5,3.5)--(.5,1);
     \draw[thick,rounded corners,ForestGreen] (5,2.5)--(2.5,2.5)--(2.5,1);
     \draw[thick,rounded corners,ForestGreen] (5,1.5)--(4.5,1.5)--(4.5,1);
     \end{tikzpicture}
\hspace{2em}
u^{(1)} \colon
  \begin{tikzpicture}[x=1.5em,y=1.5em]
    \draw[step=1,gray, thin] (0,1) grid (5,6);
      \draw[color=black, thick](0,1)rectangle(5,6);
     \draw[thick,rounded corners,ForestGreen] (5,5.5)--(3.5,5.5)--(3.5,1);
     \draw[thick,rounded corners,ForestGreen] (5,4.5)--(2.5,4.5)--(2.5,1);
     \draw[thick,rounded corners,ForestGreen] (5,3.5)--(.5,3.5)--(.5,1);
     \draw[thick,rounded corners,ForestGreen] (5,2.5)--(1.5,2.5)--(1.5,1);
     \draw[thick,rounded corners,ForestGreen] (5,1.5)--(4.5,1.5)--(4.5,1);
     \end{tikzpicture}
\hspace{2em}
u^{(2)} \colon
  \begin{tikzpicture}[x=1.5em,y=1.5em]
    \draw[step=1,gray, thin] (0,1) grid (5,6);
      \draw[color=black, thick](0,1)rectangle(5,6);
     \draw[thick,rounded corners,ForestGreen] (5,5.5)--(3.5,5.5)--(3.5,1);
     \draw[thick,rounded corners,ForestGreen] (5,4.5)--(1.5,4.5)--(1.5,1);
     \draw[thick,rounded corners,ForestGreen] (5,3.5)--(2.5,3.5)--(2.5,1);
     \draw[thick,rounded corners,ForestGreen] (5,2.5)--(.5,2.5)--(0.5,1);
     \draw[thick,rounded corners,ForestGreen] (5,1.5)--(4.5,1.5)--(4.5,1);
     \end{tikzpicture}.
\end{equation}
The diagram of the first bumpless pipe dream of \eqref{eq:bpd_ex} consists of the diagram of the bumpless pipe dream for $v$ together with the cell $(r,s)=(4,3)$. The diagram of the second bumpless pipe dream of \eqref{eq:bpd_ex} is that of $u^{(2)}$, while the diagram of the third is that of $u^{(1)}$.

We will use a diagrammatic interpretation of transition, described by Knutson and Yong in~\cite[Section 2]{Knutson.Yong}.
For $w \in S_n$, the \textbf{maximal corner} of $w$ is the lexicographically maximal cell $(r,s)$ in $D_w$.
Amongst the $\bullet$'s in $D_w$ that are northwest of the maximal corner, we call the ones that are maximally southeast \textbf{pivots}.
For $(i,j)$ a pivot of $w$, the \textbf{marching operation} is a two-step procedure on $D_w$.
First remove the lines emanating from the $\bullet$ at $(i,j)$.
Next, for every cell in $D_w$ in the rectangle with corners $(i,j)$ and $(r,s)$, move that cell strictly to the northwest in the unique way such that each cell fills a position vacated either by the removed lines or by another cell.
The resulting diagram  is $D_u$ for some $u \in S_n$, and we say $w \xrightarrow{i} u$.
The following lemma is implicit in~\cite[Section 2]{Knutson.Yong}.
\begin{lemma}
\label{l:marching}
Let $w \in S_n$ with maximal corner $(r,s)$ and $v = wt_{rw^{-1}(s)}$.
Then the pivots of $w$ are $\{(i,w(i)): i \in I(v,r)\}$ and
\[
\Phi(v,r) = \{u^{(i)}: w \xrightarrow{i} u^{(i)} \ \mbox{for}\ i \in I(v,r)\}.
\]
\end{lemma}

\subsection{Gr\"obner bases}

Recall $R=\mathbb C[Z]$.  A {\bf monomial order} is a linear ordering on monomials in $R$ such that, for any monomials ${\sf m}$, ${\sf n},$ and ${\sf p}$, we have
\begin{itemize}
\item  ${\sf m}<{\sf n}$ if and only if ${\sf m p}<{\sf np}$ and
\item  ${\sf m}\leq{\sf mp}$.  
\end{itemize}

Fix a monomial order on $R$.  
Given $f\in R$ its {\bf initial term} $\init{f}$ is the term whose monomial is largest with respect to the order.  
For a set of polynomials $F$, we define $\init{F}=\{\init{f}:f\in F\}$.  
If $F$ is an ideal, then $\init{F}$ is called the {\bf initial ideal} of $F$.  If $X=\Spec(R/I)$, the {\bf initial scheme} $\init{X}$ is $\Spec(R/\init{I})$.

  A {\bf diagonal} term order on R is a monomial order so that the initial term of any minor of $Z$ is the product of the entries on its main diagonal.  
  An {\bf antidiagonal} term order is a monomial order so that the initial term of any minor of $Z$ is the product of the entries on its main antidiagonal.

A {\bf Gr\"obner basis} of an ideal $I$ is a subset $G$ such that $\init{G}=\init{I}$.  
If $G$ is a Gr\"obner basis for $I$, then $I=\langle G\rangle$.  
Moreover, every ideal $I\subseteq R$ admits a finite Gr\"obner basis.  A Gr\"obner basis for a diagonal (resp.\ antidiagonal) term order is called a {\bf diagonal} (resp.\ {\bf antidiagonal}) Gr\"obner basis.  A subset $G$ of an ideal $I$ is a {\bf universal} Gr\"obner basis if it is a Gr\"obner basis for $I$ with respect to all monomial orders.

\subsection{Equivariant cohomology}

We need some basic notions of equivariant cohomology. Although in general, equivariant cohomology can be quite complicated, in our setting it is easy to describe axiomatically.  We will recall the properties that we will use. For an elementary but more thorough introduction to equivariant cohomology, see \cite{Macpherson}.

Consider the algebraic torus $T \subset \GL_n(\mathbb{C})$ of invertible diagonal matrices and its Lie algebra $\mathfrak{t}$ of all $n \times n$ diagonal matrices. There is a natural left action of $T \times T$ on $\Spec R$ given by scaling rows and columns separately:
\[
(t, \tau) \cdot M = t M \tau^{-1}.
\]

Now, $\Spec R$ has a $(T \times T)$-equivariant cohomology ring $H_{T \times T}(\Spec R)$. Since $\Spec R$ is contractible, we have from the definition of equivariant cohomology that 
\[
H_{T \times T}(\Spec R) \cong H_{T \times T}({\rm pt}) \cong \mathcal{O}(\mathfrak{t} \otimes \mathfrak{t}) \cong \mathbb{Z}[x_1, \dots, x_n, y_1, \dots, y_n].
\]
Every setwise-stable subscheme $X \subseteq \Spec R$ has an equivariant class $[X]_{T \times T}$, which under the above correspondence, we may identify with an integral polynomial in $2n$ variables. For $\mathcal B \subset [n] \times [n]$, let $C_ \mathcal B$ be the coordinate subspace $\Spec(R / \langle z_{ij} : (i,j) \notin  \mathcal B \rangle)$. 

For our purposes, it is enough to note that equivariant classes in $H_{T \times T}(\Spec R)$ satisfy the following three properties:

\noindent \textbf{Normalization:} For any coordinate subspace $C_{\mathcal{B}}$, we have
\[
[C_{\mathcal B}]_{T \times T}=\prod_{(i,j)\in \mathcal B} (x_i-y_j).
\]
\noindent \textbf{Additivity:}
For any $X \subseteq \Spec R$,
\[
[X]_{T \times T}=\sum_{j} {\rm mult}_{X_j}(X)\,[X_j]_{T \times T},
\]
where the sum is over the top-dimensional components of $X$ and ${\rm mult}_Y( X)$ denotes the multiplicity of $X$ along the reduced irreducible variety $Y$.
In particular, if $X=\bigcup_{i=1}^m X_i$, is a reduced scheme, then
\[
[X]_{T \times T}=\sum_{j}\ [X_j]_{T \times T},
\]
summing over components $X_j$ with $\dim X_j = \dim X$.

\noindent \textbf{Degeneration:}
If $\init{X}$ is a Gr\"obner degeneration of $X$ with respect to any term order, then \[
[X]_{T \times T}= [\init{X}]_{T \times T}.
\]

For any $X$ and any term order, we have by definition that $\init{X}$ is cut out of $\Spec R$ by a monomial ideal. Hence, $\init{X}$ is always a (schemy) union of coordinate subspaces, and its equivariant class may be computed by Additivity and Normalization. Thus, the equivariant class of any $X \subseteq \Spec R$ may be computed from these three properties, given a sufficiently good description of the initial scheme $\init{X}$.

One of the key results of \cite{Knutson.Miller} is the following.

\begin{theorem}[{\cite[Theorem~A]{Knutson.Miller}}]
\label{t:KM}
	For any permutation $w$, the matrix Schubert variety $X_w$ satisfies
	\[
	[X_w]_{T \times T}=\Schub_w(\xx;\mathbf{y}).
		\]
\end{theorem}

\section{CDG generators}
\label{s:cdg}

For $\lambda$ an integer partition, let $Z^\lambda$ be the matrix obtained from $Z=(z_{ij})$ by specializing $z_{ij}$ to $0$ for all $(i,j)\in \lambda$. 
The {\bf dominant part} of the Rothe diagram $D_w$ is the set 
\[
{\rm Dom}(w)=\{(i,j)\in D_w:r_w(i,j)=0\}.
\]
 The cells of ${\rm Dom}(w)$ make up the Young diagram of a partition $\lambda$ and we identify ${\rm Dom}(w)$ with this partition.
Define $\Esshat(w) \coloneqq \Ess(w) - \Dom(w)$.
For example, with $w = 42153$ we have ${\rm Dom}(w)  = \{ (1,1), (1,2), (1,3), (2,1) \}$, which we identify with the partition $(3,1)$.  Furthermore,   $\Ess'(w) = \{(4,3)\}$ and
\[
Z^{(3,1)}_{[4],[3]} = \begin{bmatrix}
	0 & 0 & 0 \\
	0 & z_{22} & z_{23} \\
	z_{31} & z_{32} & z_{33} \\
	z_{41} & z_{42} & z_{43}  
\end{bmatrix}.
\]

Let \[
G'_w = \bigcup_{(i,j)} \left\{ \text{minors of size $r_w(i,j)+1$ in } Z^{{\rm \Dom}(w)}_{[i],[j]} \right\},
\]
where the union is over cells $(i,j) \in \Esshat(w)$.
Then $I_w$ is generated by 
\[
G_w = G'_w \cup \{ z_{ij} : (i,j) \in \Dom(w) \}. 
\]
We call this set $G_w$ of polynomials the {\bf CDG generators} of the Schubert determinantal ideal $I_w$ (after the authors of \cite{Conca.DeNegri.Gorla} who studied similar generators in a related context). 
We are interested in when $G_w$ is a diagonal Gr\"obner basis for $I_w$; in this case, we say that $w$ and $I_w$ are {\bf CDG}. 
Note that if $w$ is CDG, then $\init{I_w}$ is reduced, since the initial terms of the polynomials in $G_w$ are all squarefree.

\begin{example}
Let $w = 42153$. Then the CDG generators of $I_w$ are
\begin{equation}\label{eq:CDG_ex}
z_{11}, z_{12}, z_{13}, z_{21},  z_{22}z_{33}z_{41} + z_{23}z_{31}z_{42} - z_{22}z_{31}z_{43} - z_{23}z_{32}z_{41}.
\end{equation}
Notice that this generating set is much smaller than the corresponding set of Fulton generators from~\eqref{eq:Fulton_ex}.
\end{example}

In $S_4$, one can check that all permutations are CDG.  In $S_5$, $13254$ and $21543$ are the only permutations which are not CDG.  Notice in particular that ${\rm Dom}(13254)=\emptyset$, so the CDG generators are simply the Fulton generators in this case.

We now observe a special class of permutations that are CDG, and indeed whose CDG generators are universal Gr\"obner.

\begin{proposition}
\label{prop:universalgb}
Fix $w\in  S_n$ so that there is a unique $(e_1,e_2)\in \Esshat(w)$.  
Furthermore, assume that $r_w(e_1,e_2)=\min\{e_1,e_2\}-1$.  
Then $G_w$ is a universal Gr\"obner basis for the Schubert determinantal ideal $I_w$.
\end{proposition}
\begin{proof}
Fix a monomial order on $R$.
It follows from \cite[Theorem~4.2]{Conca.DeNegri.Gorla} that the maximal minors of $Z^{{\rm Dom}(w)}_{[e_1],[e_2]}$ are a universal Gr\"obner basis for the ideal they generate.  
Since elements of $\{z_{ij}:(i,j)\in \Dom(w)\}$ and $G_w'$ share no variables, concatenating these two sets produces a universal Gr\"obner basis for $I_w$.
\end{proof}

\begin{example}
Continuing our running example, let $w=42153$.  Since $|\Esshat(w)|=\{(4,3)\}$ and $r_w(4,3)=2$, the generators in Equation~(\ref{eq:CDG_ex}) are a universal Gr\"obner basis for $I_w$ by Proposition~\ref{prop:universalgb}.
\end{example}

In the remainder of the paper, all term orders are assumed to be diagonal, unless otherwise specified.

\section{Block sum construction}
\label{s:block}

In this section, we define a construction that builds a partial permutation out of two partial permutations.
Its existence is encoded in the following lemma.
\begin{lemma}
\label{l:block}
Let $u$ and $v$ be partial permutations.
There is a unique $w \in S_\infty$ so that
\[
D(w) = \ytableausetup{boxsize=3.0em}
\begin{ytableau}
\text{\tiny{all boxes}}  & D(v)\\
 D(u)
 \end{ytableau}.
\]
\end{lemma}

\begin{proof}
If we can construct a partial permutation $w$ with $D(w)$ as desired, there is a unique way to extend this partial permutation to $\tilde{w} \in S_\infty$.
Since elements of $S_\infty$ are determined by their diagrams, this implies that such a $\tilde{w}$ is unique.

Viewing $u$ and $v$ as partial permutation matrices, let $w^{(1)} = \left[\begin{array}{c|c}
 0 & v\\
 \hline
 u & 0
 \end{array}\right]$.
 Note that
 \[
 D(w^{(1)}) =  \ytableausetup{boxsize=3.0em}
\begin{ytableau}
\text{\tiny{all boxes}}   & D(v)\\
 D(u) & D^{(1)}
 \end{ytableau},
 \]
 where $D^{(1)}$ is some subdiagram.
 If $D^{(1)}$ is empty, we see $w^{(1)}$ is the desired partial permutation.
Otherwise, let $(i,j)$ be the minimal cell of $D^{(1)}$ (lexicographically).
Then the $i$th row and $j$th column of $w^{(1)}$ are zero vectors, so we can construct a new partial permutation $w^{(2)} = w^{(1)} + E(i,j)$ where $E(i,j)_{kl} = \delta_{(i,j) = (k,l)}$.
Note that
\[
 D(w^{(2)}) =  \ytableausetup{boxsize=3.0em}
\begin{ytableau}
\text{\tiny{all boxes}}   & D(v)\\
 D(u) & D^{(2)}
 \end{ytableau},
\]
where $D^{(2)} \subsetneq D^{(1)}$ and the $i$th row of $D(w^{(2)})$ is empty.
Since $D^{(1)}$ has finitely many rows, by iterating this procedure we can remove every cell of $D^{(1)}$ to obtain some partial permutation $w$ with the desired diagram.
\end{proof}

\begin{definition}
\label{d:block}
	Given partial permutations $u$ and $v$ with $n$ and $m$ rows respectively, the {\bf block sum} of $u$ and $v$, denoted $u \boxplus v$, is the unique partial permutation with $n+m$ rows constructed as in Lemma~\ref{l:block}. 
\end{definition}

For $u$ and $v$ partial permutations, we can easily understand many  properties of $u \boxplus v$ in terms of $u$ and $v$.
For our purposes,  we  need to understand how block sums interact with bumpless pipe dreams and Gr\"obner bases for associated Schubert determinantal ideals.

\begin{lemma}
	\label{l:block_bpd}
	For $u,v$ partial permutations and $w = u \boxplus v$, there is a bijection from $\bpd u \times \bpd v$ to $\bpd w$ mapping the pair $(B_u,B_v)$ to a bumpless pipe dream $B_w$ satisfying
	\[
D(B_w) = \ytableausetup{boxsize=3.0em}
\begin{ytableau}
\text{\tiny{all boxes}}  & D(B_v)\\
 D(B_u)
 \end{ytableau}.
\]
\end{lemma}

\begin{proof}
Recall that the Rothe bumpless pipe dream $P_\pi$ of a permutation $\pi$ is the unique bumpless pipe dream satisfying $D(P_\pi) = D(\pi)$.
Our bijection will map $(P_u,P_v)$ to $P_w$.
Note that
\begin{equation}
\label{eqn:blockrotheBPD}
P_w = \ytableausetup{boxsize=3.2em}
\begin{ytableau}
\text{\tiny{all boxes}}  & P_v\\
 P_u & \none[\text{\tiny{\ only wires}}]	
 \end{ytableau}.
\end{equation}

Let $\phi_u$  and $\phi_v$ be sequences of droop moves satisfying $\phi_u(P_u) = B_u$ and $\phi_v(P_v) = B_v$.
We obtain $B_w$ from $P_w$ by applying $\phi_u$ and $\phi_v$ to the copies of $P_u$ and $P_v$ in $P_w$.
By Proposition~\ref{p:droop}, this map is well-defined and injective.
To see it is surjective, observe that the ``all boxes" region of $P_w$ is invariant under droop moves and prevents droop moves from occurring outside of the regions containing $P_u$ and $P_v$.
\end{proof}

We remark that the bijection in Lemma~\ref{l:block_bpd} is equivalent to mapping $(B_u,B_v)$ to the bumpless pipe dream obtained by replacing $P_u$ with $B_u$ and $P_v$ with $B_v$ in (\ref{eqn:blockrotheBPD}).

\begin{example}
Let $u=\begin{bmatrix} 0 & 1 & 0  \\ 1 & 0 & 0  \\ 0 & 0 & 0  \end{bmatrix}$, $v=\begin{bmatrix} 1 & 0 \\ 0 & 0 \end{bmatrix}$, and $w$ be the permutation associated to the partial permutation $u\boxplus v$.  The permutation matrix for $w$ is
\[ \begin{bmatrix}
 0 & 0 & 0 & 1 & 0 & 0 \\
 0 & 0 & 0 & 0 & 0 & 1 \\ 
0 & 1 & 0 & 0 & 0 & 0 \\ 
1 & 0 & 0 & 0 & 0 & 0 \\ 
0 & 0 & 0 & 0 & 1 & 0 \\ 
0 & 0 & 1 & 0 & 0 & 0 \end{bmatrix}.\]
Taking the bumpless pipe dreams for $u$ and $v$ pictured below,
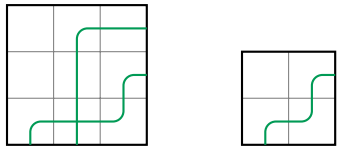
\[\begin{tikzpicture}[x=1.5em,y=1.5em]
  \draw[step=1,gray, thin] (0,1) grid (3,4);
      \draw[color=black, thick](0,1)rectangle(3,4);
     \draw[thick,rounded corners,ForestGreen] (3,2.5)--(2.5,2.5)--(2.5,1.5)--(.5,1.5)--(.5,1);
   \draw[thick,rounded corners,ForestGreen] (3,3.5)--(1.5,3.5)--(1.5,1);
     \end{tikzpicture}
\hspace{3em}
\begin{tikzpicture}[x=1.5em,y=1.5em]
  \draw[step=1,gray, thin] (0,1) grid (2,3);
      \draw[color=black, thick](0,1)rectangle(2,3);
     \draw[thick,rounded corners,ForestGreen] (2,2.5)--(1.5,2.5)--(1.5,1.5)--(.5,1.5)--(.5,1);
     \end{tikzpicture}\]
we glue and obtain a bumpless pipe dream for $w$:
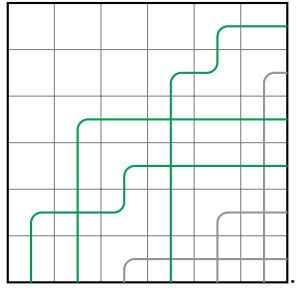
\[ \begin{tikzpicture}[x=1.5em,y=1.5em]
  \draw[step=1,gray, thin] (0,1) grid (6,7);
      \draw[color=black, thick](0,1)rectangle(6,7);
 \draw[thick,rounded corners,Gray] (6,1.5)--(2.5,1.5)--(2.5,1);
 \draw[thick,rounded corners,Gray] (6,2.5)--(4.5,2.5)--(4.5,1);
 \draw[thick,rounded corners,Gray] (6,5.5)--(5.5,5.5)--(5.5,1);
 \draw[thick,rounded corners,ForestGreen] (6,3.5)--(2.5,3.5)--(2.5,2.5)--(.5,2.5)--(.5,1);
 \draw[thick,rounded corners,ForestGreen] (6,4.5)--(1.5,4.5)--(1.5,1);
 \draw[thick,rounded corners,ForestGreen] (6,6.5)--(4.5,6.5)--(4.5,5.5)--(3.5,5.5)--(3.5,1);
     \end{tikzpicture}.\]
Pipes in the Rothe pipe dream $P_w$ that cannot be modified by droops are pictured in gray.
\end{example}

For $z_{ij}$ an indeterminate, let 
\[\downshift_a(z_{ij}) :=\begin{cases} z_{i+a\ j} & \text{if} \enspace i+a\leq m \\
 0 &\text{otherwise}
\end{cases}
\quad \text{and} \quad 
\rightshift_b(z_{ij}) :=\begin{cases} z_{i\ j+b} & \text{if} \enspace j+b\leq n \\
 0 &\text{otherwise.}
\end{cases}
\]
Extend these operators to act indeterminate-by-indeterminate on monomials, linearly on polynomials and pointwise on sets of polynomials.  
\begin{lemma}
\label{l:block_grobner}
	Let $u$ and $v$ be partial permutations such that $u$ has $b$ columns and $v$ has $a$ rows.
	If $F_u$ and $F_v$ are Gr\"obner bases of the Schubert determinantal ideals $I_u$ and $I_v$, then 	
\begin{equation}
\label{eq:block_g}
	\downshift_a(F_u)\ \cup\  \rightshift_b(F_v)\ \cup\ \{z_{ij}: 1 \leq i \leq a,\ 1 \leq j \leq b\}.
\end{equation}
	is a Gr\"obner basis for $I_{u \boxplus v}$.
\end{lemma}

\begin{proof}
Note that 
\[
I_{u \boxplus v} = \langle\downshift_a (I_u)\rangle + \langle\rightshift_b (I_v)\rangle + \langle z_{ij}: 1 \leq i \leq a,\ 1 \leq j \leq b\rangle.
\]
The result then follows from Buchberger's criterion~\cite[Theorem 2.6.6]{Cox.Little.OShea}, since the greatest common divisor of any two polynomials from different sets in Equation~\eqref{eq:block_g} is $1$.
\end{proof}

\begin{corollary}
\label{c:block_cdg}
	Let $u$ and $v$ be CDG partial permutations.
	Then $u \boxplus v$ is CDG.
\end{corollary}

\section{Monomial ideal recurrences for block predominant permutations}
\label{s:monomial}

\subsection{Predominant permutations}
We say that a partial permutation $w$ is {\bf predominant} if there is a partition $\lambda = (\lambda_1,\dots,\lambda_k)$ so that
\[
c(w) = (\lambda_1,\dots,\lambda_k,0,\dots,0,\ell,0,\dots) = \lambda 0^h \ell,
\]
for some $h, k, \ell \in \mathbb{Z}_{\geq 0}$. 
Note that we allow $h$ to be zero, so $\ell$ can immediately follow $\lambda_k$, even if $\lambda_k < \ell$.  We say a partial permutation is {\bf copredominant} if it is the transpose of a predominant permutation.
A partial permutation $w$ is {\bf block predominant} if it is a block sum of finitely many predominant partial permutations, i.e.\ if there exist predominant partial permutations $u^{(1)},\dots,u^{(k)}$ so that
\[
w = u^{(1)} \boxplus  \dots \boxplus u^{(k)}.
\]
A predominant permutation is \textbf{indecomposable} if it cannot be written as the permutation associated to a block sum of predominant partial permutations.  A predominant partial permutation is \textbf{indecomposable} if its associated permutation is indecomposable.
Note that only identity permutations are simultaneously dominant and indecomposable.

We now establish some notation that will be used for the remainder of this section.
Fix an indecomposable predominant permutation $w$.
Let 
\[
\lambda = (\lambda_1 \geq \dots \geq \lambda_r) = (m_1^{\ell_1},\dots,m_k^{\ell_k}) = {\rm Dom}(w),
\]
$\rho_i = \ell_1 + \dots + \ell_i$, and $(r,s)$ be the maximal corner in $w$.
Here, we allow $m_k$ to be zero and choose $\ell_k$ so that $r = \rho_k+1$.
Since $w$ is indecomposable, $m_1 < s$.
The pivots of $w$ are
\[
\mathcal{P}(D_w) = \{(\rho_i,m_i + \ell_i): 1 \leq i \leq k\}.
\]
For each $i$, let $w \xrightarrow{\rho_i} u^{(i)}$ by diagram marching.
By Theorem~\ref{t:LS_tree}, for $v = w t_{rw^{-1}(s)}$ we have $\Phi(v,r) = \{u^{(1)},\dots,u^{(k)}\}$ and
\[
\Schub_w = (x_r - y_s) \Schub_v + \sum_{i=1}^k \Schub_{u^{(i)}}.
\]

We now describe the diagrams of the permutations $u^{(i)}$ arising via transition.
\begin{lemma}
\label{lem:transition_essential}
For $1 \leq i \leq k$, let $S_i = \{s': (r,s') \in D_w\ \mbox{and}\ m_i + \ell_i < s' < s\}$. Then \
\[
D_{u^{(i)}} = \left(D_w \setminus \{(r,s'): s' \in S_i\}\right) \cup \{(\rho_i,m_i + \ell_i)\} \cup \{(\rho_i,s'):s' \in S_i\}.
\]
\end{lemma}

\begin{proof}
	The cells removed are precisely those that must be moved by the marching operation.
	Since $w$ is predominant, the only vacated cells are the pivot and those to the right of the pivot that are not crossed out, as described above.
\end{proof}

As a consequence, the class of block predominant permutations is closed under transition.
\begin{corollary}
\label{cor:predom_transition}
Let $\pi$ be a block predominant permutation with maximal corner $(r,s)$ and $\Phi(\pi t_{r\pi^{-1}(s)},r) = \{\tau^{(1)},\dots,\tau^{(m)}\}$.
Then each $\tau^{(i)}$ is block predominant.
\end{corollary}

\begin{proof}
Write $\pi = w \boxplus \pi'$ where $w$ is an indecomposable predominant permutation.
Then each $\tau^{(i)} = u^{(i)} \boxplus \pi'$.
From Lemma~\ref{lem:transition_essential}, we see each $u^{(i)}$ is block predominant.
\end{proof}

\subsection{Monomial ideal recurrences}
For $w$ a permutation, recall that $G_w$ is the set of CDG generators for $I_w$.
We define the monomial ideal
\[
J_w = \langle \init g: g \in G_w \rangle.
\]
Note that, by definition, $G_w$ is a Gr\"obner basis for $I_w$ if and only if $J_w = \init {I_w}$.
Similarly, let $J^\lambda_{ij}$ be the ideal generated by initial terms of non-zero maximal minors in the matrix $Z^\lambda_{[i],[j]}$.
By construction,
\begin{equation}
\label{eq:fake_initial}
J_w = \langle z_{ij}: (i,j) \in \Dom(w)\rangle + \sum_{(i,j) \in \Ess'(w)} J^{{\rm Dom}(w)}_{ij}.
\end{equation}

We will show that transition gives a recurrence on the monomial ideals $J_w$ for block predominant permutations.
To do so, we first prove some technical lemmas about the ideals $J_w$.
Consulting Example~\ref{ex:transition} may help clarify these results and their proofs.

\begin{lemma}
\label{lem:containment}
For $w$ a block predominant permutation, $J_v \subseteq J_w$ and $J_v \subseteq J_{u^{(i)}}$ for each $i$.
\end{lemma}

\begin{proof}
The containment $J_v \subseteq J_w$ is immediate, since we have $D_v \subseteq D_w$.
For the containment $J_v \subseteq J_{u^{(i)}}$, let $\Ess'(v) = \{(r,p_1),\dots,(r,p_h)\}$ with $p_1 < \dots < p_h$ and 
\[
\Ess(v) \setminus \Ess(u^{(i)}) = \{(r,p_g),\dots,(r,p_h)\} \subseteq \{(r,m_j): j > i\} \cup \{(r,s-1)\}.
\]
Fix $p \in \{p_g,\dots,p_h\}$ and let ${\sf t}$ be the initial term of a minor of size $q \coloneqq r_v(r,p)+1$ in $Z^{{\rm Dom}(v)}_{[r],[p]}$.
For $g \neq 1$, we see $(r,p_{g-1}) \in \Ess(u^{(i)})$ and
\[
q = r_{u^{(i)}}(r,p_{g-1}) + 1 + r_{u^{(i)}}(\rho_i,p) + 1.
\]
By the pigeonhole principle, this implies ${\sf t}$ is divisible by an initial term of a generator from one of $(r,\rho_{g-1}), (\rho_i,p) \in \Ess(u^{(i)})$.
When $g = 1$, every choice of ${\sf t}$ is the product of a fixed term and an initial term of a generator from $(\rho_i,p) \in \Ess(u^{(i)})$.
\end{proof}

\begin{lemma}
\label{lem:good_ideal}

For $w$ a block predominant permutation, we have $J_w/J_v = z_{rs} J^\lambda_{r{-1}\ s{-1}}/J_v$.

\end{lemma}

\begin{proof}
By restricting to the block containing the maximal corner, we may assume $w$ is an indecomposable predominant permutation.
By Lemma~\ref{lemma:diagramtransition}, we have $D_w =  D_v \cup \{(r,s)\}$.
Then 
\[
\Ess(w) \setminus \Ess(v) = \{(r,s)\},
\]
so $J_w/J_v = J_{rs}^\lambda/J_v$.

Note that $z_{rs}J_{r-1\ s-1}^\lambda/J_v \subseteq J_w/J_v$ by definition.
We will show the reverse containment, which says that $J_{rs}^\lambda/J_v \subseteq z_{rs} J^\lambda_{r-1\ s-1}/J_v$.

We explicitly prove the case $r \leq s$.
The case where $r > s$ follows from a similar argument.
Here, the generators of $J_{rs}^\lambda$ have the form ${\sf m} = z_{1j_1}\dots z_{rj_r}$.
If $j_k = s$, then ${\sf m} \in z_{rs}J_{r-1\ s-1}^\lambda$.
Otherwise, $j_r < s$.
Let 
\[
{\sf m}' = \prod_{i: j_i \leq j_r} z_{i j_i}
\]
so ${\sf m}'\mid {\sf m}$.
We claim $(r,j_r) \in D_v$, which implies ${\sf m}'$ is one of the defining generators of $J_v$, and hence the result.

To prove the claim, note that ${\sf m}'$ is the initial term of a rank $\deg({\sf m}')$ minor.
Let $p$ be maximal such that $\lambda_p \geq j_r$ and $q$ be the multiplicity of $\lambda_{p+1}$ in $\lambda$.
Necessarily, $q+1$ columns in the minor corresponding to ${\sf m}'$ must be in the set $\{\lambda_{p+1}+1,\dots,\lambda_p\}$.
The $(q+1)$st such column is $j_r$ by the definition of ${\sf m}'$.
Observe that $v(p+i) = \lambda_{p+1} + i$ for $i \in \{1,\dots,q\}$.
By exhaustion, $j_r > \lambda_{p+1}+q$, so $(r,j_r) \in D_v$,  completing the proof.
\end{proof}

We now establish the key recurrence on the monomial ideals $J_w$ for block predominant permutations.

\begin{theorem}
\label{thm:predom_transition}
Let $w$ be a block predominant permutation.
Then
\begin{equation}
\label{eq:ideal_recurrence}
J_w = (J_v + \langle z_{rs} \rangle) \cap \left( \bigcap J_{u^{(i)}} \right).
\end{equation}
\end{theorem}

\begin{proof}
Note that the maximal corner always occurs in the bottom left block of $w$, so we can assume $w$ is predominant.
By Lemma~\ref{lem:containment}, we see $J_v$ is contained in both sides of Equation~(\ref{eq:ideal_recurrence}).
Therefore, we need only prove
\[
J_w/J_v = (J_v + \langle z_{rs} \rangle)/J_v \cap \left( \bigcap_{i=1}^k J_{u^{(i)}}/J_v \right).
\]
By Lemma~\ref{lem:good_ideal}, we see $J_w/J_v = z_{rs} J^\lambda_{r-1\ s-1}/J_v$, while $(J_v + \langle z_{rs} \rangle)/J_v = \langle z_{rs} \rangle/J_v$.
Therefore, the following claim will imply our result:

\ 

\noindent \textbf{Claim:} For $q \leq k$, we have 
\[
\bigcap_{i=1}^q J_{u^{(i)}}/J_v = J^\lambda_{\rho_q s-1}/J_v.
\]

We prove the claim by induction on $q$.
For the base case $q=1$, recall $\rho_1 = \ell_1$ and observe that $\Ess(u^{(1)})\setminus \Ess(v) = \{(\rho_1,s)\}$, so $J_{u^{(1)}}/J_v = J^\lambda_{\rho_1 s-1}/J_v$.
More generally, 
\[
\Ess(u^{(j)}) \setminus \Ess(v) \subseteq \{(\rho_j,m_i): i < j)\} \cup\{(\rho_j,s-1)\}.
\]
We prove the inductive step by showing both containments.
Example~\ref{ex:transition} illustrates these arguments.

\begin{itemize}
\item[$\subseteq$] By Lemma~\ref{lem:transition_essential}, the only new essential conditions come from minors corresponding to entries $(\rho_q,m_j) \in \Ess(u^{(q)})$, which we analyze individually.
The initial terms of these minors are monomials in the defining generators of $J^\lambda_{\rho_q m_j}$, which have support in the rows $\rho_{j}+1,\dots,\rho_q$.
Since $J^\lambda_{\rho_j s-1} = \bigcap_{i=1}^j J_{u^{(j)}}$ by inductive hypothesis, we see all of our generators arising from $(\rho_q,m_j)$ are contained in $J^\lambda_{\rho_q m_j} \cap J^\lambda_{\rho_j s-1}$, which in turn is contained in $J^\lambda_{\rho_q s-1}$.
\item[$\supseteq$] Since $(\rho_q, s-1) \in \Ess(u^{(q)}) / \Ess(v)$ by Lemma~\ref{lem:transition_essential}, we see $J^\lambda_{\rho_q m_j}/J_v \subseteq J_{u^{(q)}}/J_v$ unless some additional variables $z_{\rho_q h}$ are set to zero, in which case ${\rm Dom}(v) \subsetneq {\rm Dom}(u^{(q)})$.
This only happens when $\ell_q = 1$, in which case $\rho_{q-1} = \rho_q - 1$.
In this case, all minors not involving one of these $z_{\rho_q j}$ are in $J_{u^{(q)}}/J_v$, while those involving them are found in $\langle z_{\rho_q j} \rangle \cap J^\lambda_{\rho_{q-1} s-1} \subseteq \bigcap_{i=1}^q J_{u^{(i)}}$.

\end{itemize}
This completes our proof.
\end{proof}

\begin{example}
\label{ex:transition}
Consider the permutation $w = 67341\ 10\ 2589$.
Applying the Lascoux-Sch\"utzenberger transition equations, we have
\[
v= 673419258,\quad \Phi(w,6) = \{u^{(1)} = 693417258, u^{(2)} = 673914258, u^{(3)} = 673491258\}
\]
 with diagrams
 
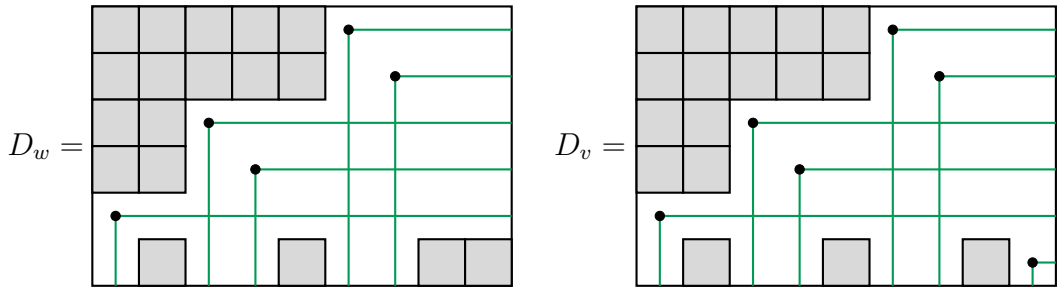
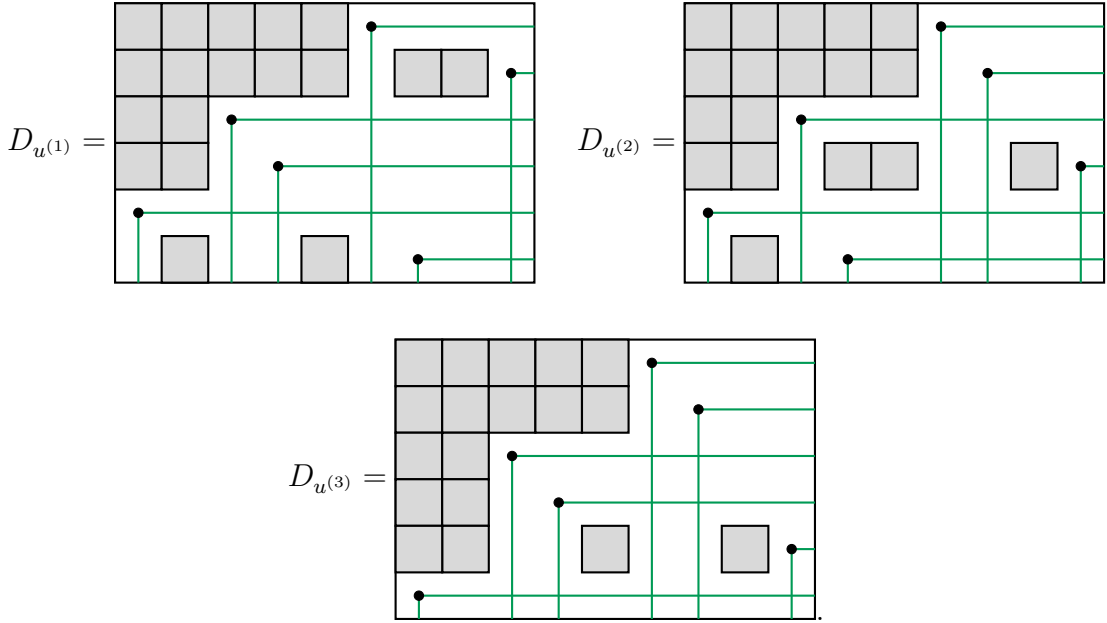
\[
  \begin{tikzpicture}[x=1.5em,y=1.5em]
      \draw[color=black, thick](0,1)rectangle(9,7);
     \filldraw[color=black, fill=gray!30, thick](0,6)rectangle(1,7);
     \filldraw[color=black, fill=gray!30, thick](0,5)rectangle(1,6);
     \filldraw[color=black, fill=gray!30, thick](0,4)rectangle(1,5);
     \filldraw[color=black, fill=gray!30, thick](0,3)rectangle(1,4);
     \filldraw[color=black, fill=gray!30, thick](1,6)rectangle(2,7);
	\filldraw[color=black, fill=gray!30, thick](1,5)rectangle(2,6);
	\filldraw[color=black, fill=gray!30, thick](1,4)rectangle(2,5);
	\filldraw[color=black, fill=gray!30, thick](1,3)rectangle(2,4);
	\filldraw[color=black, fill=gray!30, thick](2,6)rectangle(3,7);
	\filldraw[color=black, fill=gray!30, thick](2,5)rectangle(3,6);
	\filldraw[color=black, fill=gray!30, thick](3,6)rectangle(4,7);
	\filldraw[color=black, fill=gray!30, thick](3,5)rectangle(4,6);
	\filldraw[color=black, fill=gray!30, thick](4,6)rectangle(5,7);
	\filldraw[color=black, fill=gray!30, thick](4,5)rectangle(5,6);
	\filldraw[color=black, fill=gray!30, thick](1,1)rectangle(2,2);
	\filldraw[color=black, fill=gray!30, thick](4,1)rectangle(5,2);
	\filldraw[color=black, fill=gray!30, thick](7,1)rectangle(8,2);
	\filldraw[color=black, fill=gray!30, thick](8,1)rectangle(9,2);

     \draw[thick,ForestGreen] (9,6.5)--(5.5,6.5)--(5.5,1);
     \draw[thick,ForestGreen] (9,5.5)--(6.5,5.5)--(6.5,1);
     \draw[thick,ForestGreen] (9,4.5)--(2.5,4.5)--(2.5,1);
     \draw[thick,ForestGreen] (9,2.5)--(.5,2.5)--(.5,1);
     \draw[thick,ForestGreen] (9,3.5)--(3.5,3.5)--(3.5,1);
     \filldraw [black](5.5,6.5)circle(.1);
     \filldraw [black](6.5,5.5)circle(.1);
     \filldraw [black](2.5,4.5)circle(.1);
     \filldraw [black](.5,2.5)circle(.1);
     \filldraw [black](3.5,3.5)circle(.1);
\draw (-1,4) node  [align=left] {$D_w=$};
     \end{tikzpicture}
\hspace{1em}
 \begin{tikzpicture}[x=1.5em,y=1.5em]
      \draw[color=black, thick](0,1)rectangle(9,7);
     \filldraw[color=black, fill=gray!30, thick](0,6)rectangle(1,7);
     \filldraw[color=black, fill=gray!30, thick](0,5)rectangle(1,6);
     \filldraw[color=black, fill=gray!30, thick](0,4)rectangle(1,5);
     \filldraw[color=black, fill=gray!30, thick](0,3)rectangle(1,4);
     \filldraw[color=black, fill=gray!30, thick](1,6)rectangle(2,7);
	\filldraw[color=black, fill=gray!30, thick](1,5)rectangle(2,6);
	\filldraw[color=black, fill=gray!30, thick](1,4)rectangle(2,5);
	\filldraw[color=black, fill=gray!30, thick](1,3)rectangle(2,4);
	\filldraw[color=black, fill=gray!30, thick](2,6)rectangle(3,7);
	\filldraw[color=black, fill=gray!30, thick](2,5)rectangle(3,6);
	\filldraw[color=black, fill=gray!30, thick](3,6)rectangle(4,7);
	\filldraw[color=black, fill=gray!30, thick](3,5)rectangle(4,6);
	\filldraw[color=black, fill=gray!30, thick](4,6)rectangle(5,7);
	\filldraw[color=black, fill=gray!30, thick](4,5)rectangle(5,6);
	\filldraw[color=black, fill=gray!30, thick](1,1)rectangle(2,2);
	\filldraw[color=black, fill=gray!30, thick](4,1)rectangle(5,2);
	\filldraw[color=black, fill=gray!30, thick](7,1)rectangle(8,2);

     \draw[thick,ForestGreen] (9,6.5)--(5.5,6.5)--(5.5,1);
     \draw[thick,ForestGreen] (9,5.5)--(6.5,5.5)--(6.5,1);
     \draw[thick,ForestGreen] (9,4.5)--(2.5,4.5)--(2.5,1);
     \draw[thick,ForestGreen] (9,2.5)--(.5,2.5)--(.5,1);
     \draw[thick,ForestGreen] (9,3.5)--(3.5,3.5)--(3.5,1);
     \draw[thick,ForestGreen] (9,1.5)--(8.5,1.5)--(8.5,1);
     \filldraw [black](5.5,6.5)circle(.1);
     \filldraw [black](6.5,5.5)circle(.1);
     \filldraw [black](2.5,4.5)circle(.1);
     \filldraw [black](.5,2.5)circle(.1);
     \filldraw [black](3.5,3.5)circle(.1);
     \filldraw [black](8.5,1.5)circle(.1);
\draw (-1,4) node  [align=left] {$D_v=$};
     \end{tikzpicture}
\]

\[
   \begin{tikzpicture}[x=1.5em,y=1.5em]
      \draw[color=black, thick](0,1)rectangle(9,7);
     \filldraw[color=black, fill=gray!30, thick](0,6)rectangle(1,7);
     \filldraw[color=black, fill=gray!30, thick](0,5)rectangle(1,6);
     \filldraw[color=black, fill=gray!30, thick](0,4)rectangle(1,5);
     \filldraw[color=black, fill=gray!30, thick](0,3)rectangle(1,4);
     \filldraw[color=black, fill=gray!30, thick](1,6)rectangle(2,7);
	\filldraw[color=black, fill=gray!30, thick](1,5)rectangle(2,6);
	\filldraw[color=black, fill=gray!30, thick](1,4)rectangle(2,5);
	\filldraw[color=black, fill=gray!30, thick](1,3)rectangle(2,4);
	\filldraw[color=black, fill=gray!30, thick](2,6)rectangle(3,7);
	\filldraw[color=black, fill=gray!30, thick](2,5)rectangle(3,6);
	\filldraw[color=black, fill=gray!30, thick](3,6)rectangle(4,7);
	\filldraw[color=black, fill=gray!30, thick](3,5)rectangle(4,6);
	\filldraw[color=black, fill=gray!30, thick](4,6)rectangle(5,7);
	\filldraw[color=black, fill=gray!30, thick](4,5)rectangle(5,6);
	\filldraw[color=black, fill=gray!30, thick](1,1)rectangle(2,2);
	\filldraw[color=black, fill=gray!30, thick](4,1)rectangle(5,2);
	\filldraw[color=black, fill=gray!30, thick](7,5)rectangle(8,6);
	\filldraw[color=black, fill=gray!30, thick](6,5)rectangle(7,6);

     \draw[thick,ForestGreen] (9,6.5)--(5.5,6.5)--(5.5,1);
     \draw[thick,ForestGreen] (9,1.5)--(6.5,1.5)--(6.5,1);
     \draw[thick,ForestGreen] (9,4.5)--(2.5,4.5)--(2.5,1);
     \draw[thick,ForestGreen] (9,2.5)--(.5,2.5)--(.5,1);
     \draw[thick,ForestGreen] (9,3.5)--(3.5,3.5)--(3.5,1);
     \draw[thick,ForestGreen] (9,5.5)--(8.5,5.5)--(8.5,1);
     \filldraw [black](5.5,6.5)circle(.1);
     \filldraw [black](6.5,1.5)circle(.1);
     \filldraw [black](2.5,4.5)circle(.1);
     \filldraw [black](.5,2.5)circle(.1);
     \filldraw [black](3.5,3.5)circle(.1);
     \filldraw [black](8.5,5.5)circle(.1);
\draw (-1.25,4) node  [align=left] {$D_{u^{(1)}}=$};
     \end{tikzpicture}
\hspace{1em}
  \begin{tikzpicture}[x=1.5em,y=1.5em]
      \draw[color=black, thick](0,1)rectangle(9,7);
     \filldraw[color=black, fill=gray!30, thick](0,6)rectangle(1,7);
     \filldraw[color=black, fill=gray!30, thick](0,5)rectangle(1,6);
     \filldraw[color=black, fill=gray!30, thick](0,4)rectangle(1,5);
     \filldraw[color=black, fill=gray!30, thick](0,3)rectangle(1,4);
     \filldraw[color=black, fill=gray!30, thick](1,6)rectangle(2,7);
	\filldraw[color=black, fill=gray!30, thick](1,5)rectangle(2,6);
	\filldraw[color=black, fill=gray!30, thick](1,4)rectangle(2,5);
	\filldraw[color=black, fill=gray!30, thick](1,3)rectangle(2,4);
	\filldraw[color=black, fill=gray!30, thick](2,6)rectangle(3,7);
	\filldraw[color=black, fill=gray!30, thick](2,5)rectangle(3,6);
	\filldraw[color=black, fill=gray!30, thick](3,6)rectangle(4,7);
	\filldraw[color=black, fill=gray!30, thick](3,5)rectangle(4,6);
	\filldraw[color=black, fill=gray!30, thick](4,6)rectangle(5,7);
	\filldraw[color=black, fill=gray!30, thick](4,5)rectangle(5,6);
	\filldraw[color=black, fill=gray!30, thick](1,1)rectangle(2,2);
	\filldraw[color=black, fill=gray!30, thick](4,3)rectangle(5,4);
	\filldraw[color=black, fill=gray!30, thick](7,3)rectangle(8,4);
	\filldraw[color=black, fill=gray!30, thick](3,3)rectangle(4,4);

     \draw[thick,ForestGreen] (9,6.5)--(5.5,6.5)--(5.5,1);
     \draw[thick,ForestGreen] (9,5.5)--(6.5,5.5)--(6.5,1);
     \draw[thick,ForestGreen] (9,4.5)--(2.5,4.5)--(2.5,1);
     \draw[thick,ForestGreen] (9,2.5)--(.5,2.5)--(.5,1);
     \draw[thick,ForestGreen] (9,3.5)--(8.5,3.5)--(8.5,1);
	\draw[thick,ForestGreen] (9,1.5)--(3.5,1.5)--(3.5,1);
     \filldraw [black](5.5,6.5)circle(.1);
     \filldraw [black](6.5,5.5)circle(.1);
     \filldraw [black](2.5,4.5)circle(.1);
     \filldraw [black](.5,2.5)circle(.1);
     \filldraw [black](8.5,3.5)circle(.1);
     \filldraw [black](3.5,1.5)circle(.1);
     \draw (-1.25,4) node  [align=left] {$D_{u^{(2)}}=$};
     \end{tikzpicture}
\]

\[\begin{tikzpicture}[x=1.5em,y=1.5em]
      \draw[color=black, thick](0,1)rectangle(9,7);
     \filldraw[color=black, fill=gray!30, thick](0,6)rectangle(1,7);
     \filldraw[color=black, fill=gray!30, thick](0,5)rectangle(1,6);
     \filldraw[color=black, fill=gray!30, thick](0,4)rectangle(1,5);
     \filldraw[color=black, fill=gray!30, thick](0,3)rectangle(1,4);
     \filldraw[color=black, fill=gray!30, thick](1,6)rectangle(2,7);
	\filldraw[color=black, fill=gray!30, thick](1,5)rectangle(2,6);
	\filldraw[color=black, fill=gray!30, thick](1,4)rectangle(2,5);
	\filldraw[color=black, fill=gray!30, thick](1,3)rectangle(2,4);
	\filldraw[color=black, fill=gray!30, thick](2,6)rectangle(3,7);
	\filldraw[color=black, fill=gray!30, thick](2,5)rectangle(3,6);
	\filldraw[color=black, fill=gray!30, thick](3,6)rectangle(4,7);
	\filldraw[color=black, fill=gray!30, thick](3,5)rectangle(4,6);
	\filldraw[color=black, fill=gray!30, thick](4,6)rectangle(5,7);
	\filldraw[color=black, fill=gray!30, thick](4,5)rectangle(5,6);
	\filldraw[color=black, fill=gray!30, thick](1,2)rectangle(2,3);
	\filldraw[color=black, fill=gray!30, thick](4,2)rectangle(5,3);
	\filldraw[color=black, fill=gray!30, thick](7,2)rectangle(8,3);
	\filldraw[color=black, fill=gray!30, thick](0,2)rectangle(1,3);

     \draw[thick,ForestGreen] (9,6.5)--(5.5,6.5)--(5.5,1);
     \draw[thick,ForestGreen] (9,5.5)--(6.5,5.5)--(6.5,1);
     \draw[thick,ForestGreen] (9,4.5)--(2.5,4.5)--(2.5,1);
     \draw[thick,ForestGreen] (9,1.5)--(.5,1.5)--(.5,1);
     \draw[thick,ForestGreen] (9,2.5)--(8.5,2.5)--(8.5,1);
     \draw[thick,ForestGreen] (9,3.5)--(3.5,3.5)--(3.5,1);
     \filldraw [black](5.5,6.5)circle(.1);
     \filldraw [black](6.5,5.5)circle(.1);
     \filldraw [black](2.5,4.5)circle(.1);
     \filldraw [black](.5,1.5)circle(.1);
 \filldraw [black](8.5,2.5)circle(.1);
     \filldraw [black](3.5,3.5)circle(.1);
\draw (-1.25,4) node  [align=left] {$D_{u^{(3)}}=$};
     \end{tikzpicture}.\]

The dominant component in these diagrams is $\lambda = (5^2,2^2,0^1)$ except that ${\rm Dom}(u^{(3)}) = \lambda' = (5^2,2^3)$.

The monomials coming from minors corresponding to the essential boxes $(6,2)$ and $(6,4)$ are in $J_v$.
By Lemma~\ref{lem:containment}, $J_v \subseteq J_w,J_{u^{(1)}},J_{u^{(2)}},J_{u^{(4)}}$, but the reader can also check this directly.
For example, $z_{33}z_{45}z_{51}z_{62} \in J_v$ is also in $J_{u^{(1)}}$ and is divisible by $z_{51}z_{62} \in J_{u^{(2)}}$ as well as by $z_{51} \in J_{u^{(3)}}$.

We have $J_w / J_v = z_{69} J^{\lambda}_{58}  / J_v$ by Lemma~\ref{lem:good_ideal}.
By direct observation, $J_{u^{(1)}}/J_v = J^\lambda_{28}/J_v$.
To see that
$(J_{u^{(1)}} \cap J_{u^{(2)}})/J_v = J^\lambda_{48}$, observe that $J^\lambda_{48} \subseteq J_{u^{(2)}}$ by Equation~\eqref{eq:fake_initial}.
The opposite containment follows since the minors coming from $(4,5) \in \Ess(u^{(2)})$ correspond to $J^\lambda_{45}$, while $(J^\lambda_{45} \cap J^\lambda_{28}) \subseteq J^\lambda_{48}$.
Next, we show $(J_{u^{(1)}} \cap J_{u^{(2)}} \cap J_{u^{(3)}})/J_v = J^\lambda_{58}/J_v$.
To show the forward containment, we study $(5,2),(5,4),(5,8) \in \Ess(u^{(3)})$ individually.
We have
\[
J^\lambda_{52} \cap J^\lambda_{48}, J^{\lambda'}_{54} \cap J^\lambda_{28}, J^{\lambda'}_{58} \subseteq J^\lambda_{58}.
\]
Moreover, the only monomials generating $J^\lambda_{58}$ that aren't found in $J^{\lambda'}_{58}$ are those involving $z_{51}$ and $z_{52}$, so we see $J^\lambda_{58} \subseteq J^{\lambda'}_{58} + J^\lambda_{52} \cap J^\lambda_{48}$.
\end{example}

%
%

\section{Proof of main theorem}
\label{s:proof}

We will use the following lemma of Knutson and Miller, originally stated in greater generality.
Recall $R = \mathbb{C}[Z]$ where $Z = (z_{ij})_{i \in [m],j\in[n]}$.
\begin{lemma}[{\cite[Lemma 1.7.5]{Knutson.Miller}}]
	\label{l:cheating_lemma}
	Let $I \subseteq R$ be an ideal such that $\Spec R/I$ is stable under the $T \times T$ action.
	Suppose $J \subseteq I$ is an equidimensional radical ideal.
	If we have the equality $[\Spec R/I]_{T \times T} = [\Spec R/J]_{T \times T}$ of equivariant cohomology classes, then $I = J$.
\end{lemma}

Given a bumpless pipe dream $P$, write $\mathcal L_P=\langle z_{ij}:(i,j)\in D(P) \rangle.$ 

\begin{proposition}
\label{p:predominant}
	Suppose $w$ is a block predominant permutation.  Then 
	\begin{enumerate}
		\item the CDG generators are a diagonal Gr\"obner basis for $I_w$ and
		\item $\displaystyle \init{I_w}=\bigcap_{P\in\bpd{w}}\mathcal L_P$.
	\end{enumerate}
\end{proposition}
\begin{proof}
We will proceed by induction on the position of the maximal corner of $w$.  Without loss of generality, we may assume $w$ is predominant.

In the base case, $w$ is dominant and the statement is trivial.

Now assume $w$ is not dominant.  Furthermore, assume the statement holds for all block predominant permutations whose maximal corners occur lexicographically before the pivot of $w$.  

By induction, we know $J_v=\init{I_v}$ and likewise $J_{u^{(i)}}=\init{I_{u^{(i)}}}$.  
Therefore
\begin{align*}
J_w&=(\init{I_v}+\langle z_{r,s}\rangle)\cap \bigcap_{u\in \Phi(v,r)} \init{I_{u}} & \text{by Theorem~\ref{thm:predom_transition}}\\
&=\left(\bigcap_{P\in\bpd{v}}\mathcal L_P+\langle z_{r,s}\rangle\right) \cap  \bigcap_{u\in \Phi(v,r)} \bigcap_{P\in\bpd{u}}\mathcal L_P & \text{by inductive hypothesis}\\
		&=\bigcap_{P\in \bpd{w}} \mathcal L_P &\text{by Lemma~\ref{lemma:diagramtransition}}.
\end{align*}

By additivity of equivariant classes, we know that \[[\Spec R/ J_w]_{T\times T}=\left[\Spec R/\bigcap_{P\in \bpd{w}} \mathcal L_P \right]_{T \times T}=\sum_{P\in\bpd{w}} \wt(P)=\mathfrak S_{w}(\mathbf x;\mathbf y)\] (where the last equality is Theorem~\ref{thm:BPDformula}).

By Theorem~\ref{t:KM}, $[\Spec R/I_w]_{T \times T}=\mathfrak S_{w}(\mathbf x;\mathbf y)$.  By degeneration, $[\Spec R / \init{I_w}]_{T \times T}=\mathfrak S_{w}(\mathbf x;\mathbf y)$ as well.

Trivially, $J_w\subseteq \init{I_w}$.  Since $J_w$ is an intersection of primes of the same dimension, it is equidimensional and radical.  Furthermore, $\Spec R/\init{I_w}$ is stable with respect to the $T \times T$ action.  Therefore by Lemma~\ref{l:cheating_lemma}, we have $J_w=\init{I_w}$ and the proposition holds.
\end{proof}

A similar fact is true for vexillary permutations.

\begin{proposition}
\label{p:vexillary}
If $w$ is vexillary, then $w$ is CDG. Moreover, in this case, we have 
\[
 \init{I_w}=\bigcap_{P\in\bpd{w}}\mathcal L_P.
\]
\end{proposition}
\begin{proof}
Since $w$ is vexillary, the Fulton generators are diagonal Gr\"obner by \cite[Theorem~3.8]{Knutson.Miller.Yong}. If a defining minor intersects the dominant part of $D_w$, then so does its main diagonal. Hence its initial term is a multiple of that variable, and we may remove that equation from the generating set to obtain a new Gr\"obner basis \cite[Lemma~2.7.3]{Cox.Little.OShea}. Thus, the CDG generators are also diagonal Gr\"obner.

The characterization in terms of bumpless pipe dreams then follows from \cite[Theorem~4.4]{Knutson.Miller.Yong}, as interpreted via bumpless pipe dreams in \cite{Weigandt}.
\end{proof}

We say a permutation is \textbf{banner} if it is a block sum of predominant, copredominant, and vexillary partial permutations.
The following theorem is our main result, combining essentially everything else established in this paper.

\begin{theorem}
\label{t:banner}
	Let $w$ be a banner permutation. Then
	\begin{enumerate}
		\item $w$ is CDG, and 
		\item $\init {I_w} = \bigcap_{P \in \bpd w} \mathcal{L}_P.$ In particular, $\init {I_w}$ is radical.
	\end{enumerate}
\end{theorem}
\begin{proof}
For (1), Proposition~\ref{p:predominant} proves the predominant case.
The copredominant case follows by transposition, while the vexillary case is Proposition~\ref{p:vexillary}.
The result follows by Lemma~\ref{l:block_grobner}.
Meanwhile, (2) follows from Proposition~\ref{p:predominant}, Proposition~\ref{p:vexillary}, and Proposition~\ref{p:block_CDG} below.
\end{proof}

\begin{proposition}
\label{p:block_CDG}
Let $w=u^{(1)}\boxplus \cdots \boxplus u^{(k)}$ be a block sum of partial permutations. 
If \[\init{I_{u^{(i)}}}=\bigcap_{P\in \bpd{u^{(i)}}}\mathcal L_P\] for all $i\in [k]$, then \[\init{I_{w}}=\bigcap_{P\in \bpd{w}}\mathcal L_P,\]
so in particular $\init{I_w}$ is radical.
\end{proposition}
\begin{proof}
Fix partial permutations $u\in M_{p,n}$ and $v\in M_{m,q}$.  It is enough to consider the case when $w=u\boxplus v$. 

Write $\widetilde{u}$ for the partial permutation matrix obtained by prepending $m$ rows of $0$'s to $u$.  Likewise, let $\widetilde{v}$ be the partial permutation obtained by prepending $n$ columns of $0$'s to $v$.
Observe that 
\[
I_{\widetilde{u}} = \langle z_{ij}: (i,j) \in [m] \times [n] \rangle + \langle\downshift_{m}(I_u)\rangle \quad \mbox{and} \quad I_{\widetilde{v}} = \langle z_{ij}: (i,j) \in [m] \times [n] \rangle + \langle\rightshift_{n}(I_v)\rangle.
\]
Therefore $I_w = I_{\widetilde{u}}+I_{\widetilde{v}}$, so
\[
\init{I_{\widetilde{u}}}+\init{I_{\widetilde{v}}}\subseteq \init{I_w}.
\]
Since $\init{I_u}$ and $\init{I_v}$ are radical, so are $\init{I_{\widetilde{u}}}$ and $\init{I_{\widetilde{v}}}$.
In particular,
\begin{align*}
\init{I_{\widetilde{u}}}+\init{I_{\widetilde{v}}}&=\bigcap_{P_1\in \bpd{\widetilde{u}},P_2\in \bpd{\widetilde{v}}}\mathcal L_{P_1\cup P_2}\\
&= \bigcap_{P\in \bpd{w}} \mathcal L_P,
\end{align*}
with the first equality by hypothesis since $D(P_1) \cap D(P_2) = [m] \times [n]$ and the second from Lemma~\ref{l:block_bpd}.
Therefore, \[\left[\Spec R/ \bigcap_{P\in \bpd{w}} \mathcal L_P\right]_{T \times T}=\Schub_w(\mathbf x;\mathbf y).\]
Furthermore, $\init{I_{\widetilde{u}}}+\init{I_{\widetilde{v}}}$ is an equidimensional radical ideal, contained in $\init{I_w}$.  Since \[\left[\Spec R/\init{I_w}\right]_{T \times T}=\Schub_w(\mathbf x;\mathbf y)\] as well, we apply Lemma~\ref{l:cheating_lemma} to conclude \[\init{I_w}=\bigcap_{P\in \bpd{w}}\mathcal L_P\] and $\init{I_w}$ is radical.
\end{proof}

Theorem~\ref{t:banner} is a special case of Conjecture~\ref{c:main}, and provides further evidence for the general statement of the conjecture.

\section{Future directions}
\label{s:future}

Theorem~\ref{t:banner} does not exhaust the set of all CDG permutations.
For example, $w=25143$ is not banner, but one can compute that its CDG generators are a diagonal Gr\"obner basis for $I_w$.
We will conjecture a complete characterization of CDG permutations, but first must recall the notion of pattern avoidance.
For $a = a_1 \dots a_n$ a sequence of distinct integers, let $\sigma(a) = \sigma_1 \dots \sigma_n$  be the unique permutation in  $\mathcal{S}_n$ so that $\sigma_i < \sigma_j$ if and only if $a_i < a_j$ for all $i,j$.
The permutation $w = w_1 \dots w_n$ \textbf{contains} the permutation $v = v_1 \dots v_k$ if there is a subsequence  $w' = w_{i_1}\dots w_{i_k}$ of $w$ with $\sigma(w') = v$.
If $w$ does not contain $v$, then $w$ \textbf{avoids} $v$.

We conjecture the following characterization of CDG permutations.

\begin{conjecture}\label{conj:patterns}
Let $w$ be a permutation. The CDG generators are a diagonal Gr\"obner basis for $I_w$ if and only if $w$ avoids all eight of the patterns
\[
13254,21543,214635,215364,241635,315264,215634,4261735.
\]

\end{conjecture}
We checked that every permutation in $S_8$ avoiding these patterns is CDG. 
Proving that the CDG property is governed by pattern avoidance is a question of independent interest. 
A corollary of Conjecture~\ref{conj:patterns} is the following.
\begin{conjecture}\label{conj:multfree}
Suppose $\Schub_w$ is a multiplicity-free sum of monomials. 
Then the CDG generators are a diagonal Gr\"obner basis for $I_w$.
\end{conjecture}
Conjecture~\ref{conj:multfree} would follow immediately from Conjecture~\ref{conj:patterns} by the known pattern characterization of those $w$ in Conjecture~\ref{conj:multfree} \cite{Fink.Meszaros.StDizier} (see also \cite[P0055]{Tenner:database}).

If Conjecture~\ref{c:main} holds, then as a consequence, $\Spec R/\init{I_w}$ is reduced if and only if each $P\in\bpd{w}$ has a distinct diagram.  Data suggests that this condition is also governed by pattern containment (see \cite{Heck}).

\section*{Acknowledgments}
OP was partially supported by a Mathematical Sciences Postdoctoral Research Fellowship (\#1703696) from the National Science
Foundation.

We are very grateful for  helpful conversations with Patricia Klein, Allen Knutson, Jenna Raj\-chgot, David Speyer, and Alexander Yong.

%
%

\bibliographystyle{amsalpha} 
\bibliography{dominant}

\end{document}